\documentclass[12pt]{article}

\usepackage{times}

\usepackage[utf8]{inputenc}
\usepackage{commath}
\usepackage{amsthm, amscd}
\usepackage{amsmath}
\usepackage{amssymb}
\usepackage{cite}
\usepackage{setspace}

\usepackage{bigints}
\usepackage{comment}
\usepackage{mathtools}
\usepackage{mathrsfs}
\usepackage{fancyhdr}

\allowdisplaybreaks[4]

\numberwithin{equation}{section}
\usepackage{etoolbox}
\patchcmd{\thebibliography}
  {\settowidth}
  {\setlength{\itemsep}{0pt plus -10pt}\settowidth}
  {}{}
\apptocmd{\thebibliography}
  {
  }
  {}{}
\makeatletter
\newtheorem*{rep@theorem}{\rep@title}
\newcommand{\newreptheorem}[2]{%
\newenvironment{rep#1}[1]{%
 \def\rep@title{#2 \ref{##1}}%
 \begin{rep@theorem}}%
 {\end{rep@theorem}}}
\makeatother

\theoremstyle{theorem}

\newreptheorem{theorem}{Theorem}
\newtheorem{thm}{Theorem}[section]
\newtheorem*{thm*}{Theorem}
\theoremstyle{definition}
\newtheorem{prop}[thm]{Proposition}
\newtheorem*{prop*}{Proposition}
\newtheorem{defn}[thm]{Definition}
\newtheorem{lem}[thm]{Lemma}
\newtheorem{cor}[thm]{Corollary}
\newtheorem*{cor*}{Corollary}
\theoremstyle{remark}

\newtheorem{const}[thm]{Construction}
\newtheorem{rem}[thm]{Remark}

\newtheorem{notat}[thm]{Notation}

\newcounter{tmp}

\title{Analytic torsion for surfaces with cusps II. \\
Regularity, asymptotics and curvature theorem.} 

\author
{Siarhei Finski
}

\date{}

\usepackage[%
    left=1in,%
    right=1in,%
    top=1.1in,%
    bottom=0.8in,%
    paperheight=11in,%
    paperwidth=8.5in%
]{geometry}

\newcommand{\imun} {\sqrt{-1}}
\newcommand{\vol}{v}

\newcommand{\comp}{\mathbb{C}}
\newcommand{\real}{\mathbb{R}}

\newcommand{\nat}{\mathbb{N}}
\newcommand{\integ}{\mathbb{Z}}

\newcommand{\dd}{\mathbb{D}}
\newcommand{\hh}{\mathbb{H}}

\newcommand{\ccal}{\mathscr{C}}

\newcommand{\dbar}{ \overline{\partial} }

\newcommand{\laplcomp}{\Box}

\newcommand{\rk}[1]{{\rm{rk}} ( #1 )}

\renewcommand{\Re}{\operatorname{Re}}
\renewcommand{\Im}{\operatorname{Im}}
\newcommand{\scal}[2]{\big< #1, #2 \big>}

\newcommand{\modul}{\mathscr{M}}
\newcommand{\modulcomp}{\overline{\mathscr{M}}}

\newcommand{\univcurv}{\mathscr{C}}
\newcommand{\univcurvcomp}{\overline{\mathscr{C}}}

\newcommand{\td}{{\rm{Td}}}
\newcommand{\ch}{{\rm{ch}}}

\newenvironment{sciabstract}{}

\begin{document} 
\maketitle

\begin{sciabstract}
  \textbf{Abstract.} In this article we study the  Quillen norm on the determinant line bundle associated with a family of complex curves with cusps, which admit singular fibers.
  \par More precisely, we fix a family of complex curves $\pi : X \to S$, which admit at most double-point singularities. Let $(\xi, h^{\xi})$ be a holomorphic Hermitian vector bundle over $X$.
  Let $\sigma_1, \ldots, \sigma_m : S \to X$ be disjoint holomorphic sections.
  We denote the divisor $D_{X/S} := \Im(\sigma_1) + \cdots + \Im(\sigma_m)$, and endow the relative canonical line bundle $\omega_{X/S}$ with a Hermitian norm such that its restriction at each fiber of $\pi$ induces Kähler metric with hyperbolic cusps. 
  This Hermitian norm induces the Hermitian norm on the twisted relative canonical line bundle $\omega_{X/S}(D) := \omega_{X/S} \otimes \mathscr{O}_{X}(D_{X/S})$.
  \par 
  The main object of this paper is the determinant line bundle $\lambda(j^*(\xi \otimes \omega_{X/S}(D)^n))$. For $n \leq 0$, we endow it with the Quillen norm by using the analytic torsion from the first paper of this series.
  Then we study the regularity of this Quillen norm and its asymptotics near the locus of singular curves. In particular, our study applies to the family of degenerating pointed hyperbolic surfaces. 
  The singular terms of the asymptotics turn out to be reasonable enough, so that the curvature of $\lambda(j^*(\xi \otimes \omega_{X/S}(D)^n))$ is well-defined as a current over $S$. 
  We derive the explicit formula for this current, which gives a refinement of Riemann-Roch-Grothendieck theorem at the level of currents.
  This generalizes the curvature formulas of Takhtajan-Zograf and Bismut-Bost. 
  \begin{sloppypar}
		 We also explicit some general conditions under which the renormalized Quillen norm becomes continuous. As a consequence, we get some regularity results on the Weil-Petersson form over the moduli space of pointed curves.
		 Those regularity results are enough to conclude the well-known fact, originally due to Wolpert, that the Weil-Petersson volume of the moduli space of pointed curves is a rational multiple of a power of $\pi$.
  \end{sloppypar}
\end{sciabstract}

\pagestyle{fancy}
\lhead{}
\chead{II. Regularity, asymptotics and curvature theorem.}
\rhead{\thepage}
\cfoot{}


\newcommand{\Addresses}{{
  \bigskip
  \footnotesize
  \noindent \textsc{Siarhei Finski, UFR de Mathématiques, Case 7012, Université Paris Diderot-Paris 7, France.}\par\nopagebreak
  \noindent  \textit{E-mail }: \texttt{siarhei.finski@imj-prg.fr}.
}}

\tableofcontents

\section{Introduction}\label{sect_intro}
	In this article we study the Quillen norm on the determinant line bundle associated with a family of complex curves with cusps, which admit singular fibers.
	\par Let $X$ and $S$ be complex manifolds, and let $\pi : X \to S$ be a proper holomorphic map. The construction of Grothendick-Knudsen-Mumford \cite{Knudsen1976} (cf. also \cite[\S 3]{BGS3}) associates for every holomorphic vector bundle $\xi$ over $X$ the “determinant of the direct image of $\xi$" - the holomorphic line bundle over $S$, which we denote by $\det (R^{\bullet} \pi_* \xi)$. 
	In the special case when the cohomology groups $H^{\bullet}(\pi^{-1}(\cdot), \xi)$ have constant dimension, they form holomorphic vector bundles, and we have a natural isomorphism
	$\det (R^{\bullet} \pi_* \xi) = \otimes_i ( \Lambda^{\max} H^{i}(\pi^{-1}(\cdot), \xi))^{(-1)^i}$.
	If $X$ and $S$ are quasiprojective and $\pi$ is projective, by a theorem of Riemann-Roch-Grothendieck (cf. \cite[\S 7]{BorSerre}), the first Chern class of $\det (R^{\bullet} \pi_* \xi)$ is expressed as a push-forward by $\pi$ of characteristic classes of $\xi$ and the relative tangent bundle $TX/S$.
	\par Assume temporarily that $\pi$ is a Kähler fibration, i.e. it is a submersion such that there exists a closed $(1, 1)$-form on $X$ such that its restriction to the fibers of $\pi$ gives a Kähler form. In this case, Bismut-Gillet-Soulé in \cite{BGS3} gave an analytic construction of the line bundle $\lambda(j^*\xi)$, which they proved in \cite[Theorem 3.14]{BGS3} to coincide  with $(\det (R^{\bullet} \pi_* \xi) )^{-1}$.
	Now, endow the vector bundles $\xi, TX/S$ with Hermitian metrics $h^{\xi}$ and $h^{TX/S}$ over $X$. 
	As the dimension of the cohomology of the fibers might change from point to point, the holomorphic analytic torsion and the $L^2$-metric of the fibers on the line bundle $\lambda(j^* \xi)$ are not necessarily continuous. Nevertheless, Bismut-Gillet-Soulé in \cite[Theorems 1.3, 1.6]{BGS3} showed that the Quillen norm $\norm{\cdot}_Q$ on $\lambda(j^* \xi)$, defined as a product of the holomorphic analytic torsion and the $L^2$-metric of the fiber, is smooth.
	 When $\pi$ is trivial of relative dimension $1$, this metric was previously defined by Quillen in \cite{Quillen}. 
	The curvature theorem of Bismut-Gillet-Soulé \cite[Theorem 1.9]{BGS3} gives a refinement of Riemann-Roch-Grothendieck theorem by expressing the curvature of the Chern connection on $(\lambda(j^* \xi), \norm{\cdot}_Q^{2})$ as an integral over the fibers of $\pi$ of Chern forms associated with $(\xi, h^{\xi})$ and $(TX/S, h^{TX/S})$.
	\par
		Now, let $\pi : X \to S$ be a family of complex curves with ordinary singularities. 
		Let $\sigma_1, \ldots, \sigma_m : S \to X$ be disjoint holomorphic sections of $\pi$, which do not pass through singular points.
		Let the norm $\norm{\cdot}_{X/S}^{\omega}$ on the canonical line bundle $\omega_{X/S}$ (see Section \ref{sect_fnc}) over $X \setminus (\cup_i \Im(\sigma_i))$ be such that its restriction over each nonsingular fiber $X_t := \pi^{-1}(t)$, $t \in S \setminus |\Delta|$ of $\pi$ induces the Kähler metric with cusps at $\sigma_1(t), \ldots, \sigma_m(t)$ (see Definition \ref{defn_fsc}).
		The goal of this article is to study the associated Quillen norm, constructed using the analytic torsion, defined in the first article of this series \cite{FinII1}, and deduce the corresponding curvature formula.
	\par
	More precisely, let's denote by $\norm{\cdot}^W_{X/S}$ the Wolpert norm induced by $\norm{\cdot}_{X/S}^{\omega}$ (see Definition \ref{defn_wolp}) on the line bundle $\otimes_i \sigma_i^{*}(\omega_{X / S})$.
	We denote by $D_{X/S}$ the divisor on $X$ given by 
	\begin{equation}
		D_{X/S} := \Im(\sigma_1) + \cdots + \Im(\sigma_m).
	\end{equation}
	By $|\Delta| \subset S$, we denote the locus of singular fibers, and by $\Sigma_{X/S} \subset X$ the submanifold of singular points of the fibers (see Corollary \ref{cor_sigma}).
	\begin{const}\label{const_norm_div}
		For a complex manifold $Y$ and a divisor $D_0 \subset Y$, let $\norm{\cdot}^{\rm{div}}_{D_0}$ be the singular norm on $\mathscr{O}_Y(D_0)$, defined by
		\begin{equation}\label{defn_norm_D}
			\| s_{D_0} \|^{\rm{div}}_{D_0}(x) = 1,
		\end{equation}
		where $s_{D_0}$, ${\rm{div}} (s_{D_0}) = D_0$, is the canonical section of the divisor $D_0$, and $x \in Y \setminus |D_0|$. 
		\par We endow the twisted canonical line bundle 
		\begin{equation}\label{defn_tw_can}
			\omega_{X/S}(D) := \omega_{X/S} \otimes \mathscr{O}_X(D_{X/S})
		\end{equation}
		with the canonical Hermitian norm $\, \norm{\cdot}_{X/S}$ over $\pi^{-1}(S \setminus |\Delta|) \setminus |D_{X/S}|$, induced by $\norm{\cdot}^{\omega}_{X/S}$, $\norm{\cdot}^{\rm{div}}_{D_{X/S}}$.
	\end{const}
	Let $\xi$ be a holomorphic line bundle over $X$, and let $h^{\xi}$ be a Hermitian metric over $\pi^{-1}(S \setminus |\Delta|)$.
	For $n \leq 0$, we endow the line bundle $\det (R^{\bullet} \pi_* (\xi  \otimes \omega_{X/S}(D)^n))^{-1}$ with the Quillen norm 
	\begin{equation}\label{rel_quil_norm_intro}
		\norm{\cdot}_{Q} (g^{TX_t}, h^{\xi} \otimes \, \norm{\cdot}_{X/S}^{2n}),
	\end{equation}
	over $S \setminus |\Delta|$, $t \in S \setminus |\Delta|$, defined as the product of the $L^2$-norm and the analytic torsion of the fiber (see Definition \ref{defn_quil_norm}). 
	Now, we denote the norm 
		\begin{equation}\label{quil_wol_norm}
			\norm{\cdot}^{\mathscr{L}_n} := \big( \norm{\cdot}_Q (g^{TX_t}, h^{\xi} \otimes \, \norm{\cdot}_{X/S}^{2n}) \big)^{12} \otimes \big( \norm{\cdot}^W_{X/S} \big)^{-\rk{\xi}} \otimes \big( \norm{\cdot}^{\rm{div}}_{\Delta} \big)^{\rk{\xi}}
		\end{equation}				
		on the line bundle
		\begin{equation}\label{det_wol_prod}
			\mathscr{L}_n :=  \det \big( R^{\bullet} \pi_* (\xi  \otimes \omega_{X/S}(D)^n) \big)^{-12} \otimes \big(\otimes_i \sigma_i^{*}\omega_{X/S} \big)^{-\rk{\xi}} \otimes \mathscr{O}_S(\Delta)^{\rk{\xi}}.
		\end{equation}				
		Our first goal is to study - under various assumptions on the data - the regularity of $\, \norm{\cdot}^{\mathscr{L}_n}$ over $S \setminus |\Delta|$ and its singularities near $\Delta$. We show that the singularities of $\norm{\cdot}^{\mathscr{L}_n}$ are reasonable enough to define the first Chern form of $(\mathscr{L}_n, (\, \norm{\cdot}^{\mathscr{L}_n})^{2})$ as a current over $S$. 
		Then we compute this current explicitly, which gives us a refinement of Riemann-Roch-Grothendieck theorem on the level of currents. 
		In particular, when $\pi : X \to S$ is a family of hyperbolic surfaces without singular fibers and $(\xi, h^{\xi})$ is trivial, we get Takhatajan-Zograf formula \cite[Theorem 1]{TakZog}; when the metrics $h^{\xi}$, $\norm{\cdot}_{X/S}$ are smooth, i.e. there are no cusps and no degeneration of the metric near singular fibers, we get a formula of Bismut-Bost \cite[Théorème 2.2]{BisBost}.
	\par
	Now let's state precisely the different assumptions on the data, which we consider in this article.
	\par \textbf{Assumption S1}.\label{cond_sing}  The Hermitian metric $h^{\xi}$ extends smoothly over $X$; the Hermitian norm $\norm{\cdot}_{X/S}$ extends smoothly over $X \setminus |D_{X/S}|$ and it is pre-log-log of infinite order, with singularities along $D_{X/S}$ (cf. Definition \ref{defn_pre_loglog}).
	\par \textbf{Assumption S2}. We suppose that $\Delta$ has normal crossings. The Hermitian metric $h^{\xi}$  is pre-log-log with singularities along $\pi^{-1}(\Delta)$ (cf. Definition \ref{defn_pre_loglog}); the Hermitian norm $\norm{\cdot}_{X/S}$ is pre-log-log with singularities along $\pi^{-1}(\Delta) \cup D_{X/S}$.
	\par \textbf{Assumption S3}. We suppose that $\Delta$ has normal crossings. The Hermitian metric $h^{\xi}$ extends smoothly over $X$; the Hermitian norm $\norm{\cdot}_{X/S}$ is continuous over $X \setminus (\Sigma_{X/S} \cup |D_{X/S}|)$, has log-log growth with singularities along $\Sigma_{X/S} \cup |D_{X/S}|$ (cf. Definitions \ref{defn_loglog_gr_lin}), is good in the sense of Mumford on $X \setminus |D_{X/S}|$ with singularities along $\pi^{-1}(\Delta)$ (cf. Definition \ref{defn_pre_loglog}), and the coupling of $c_1(\omega_{X/S}(D), \norm{\cdot}_{X/S}^{2})$ with two smooth vertical vector fields over $X \setminus ( \Sigma_{X/S} \cup |D_{X/S}|)$ is continuous over $X \setminus ( \Sigma_{X/S} \cup |D_{X/S}|)$ and has log-log growth with singularities along $|D_{X/S}|$ (cf. Definition \ref{defn_loglog_gr}b)).
	\begin{rem}\label{rem_norm_cross_s1}
		If $\Delta$ has normal crossings, then, trivially, \textbf{S1} implies both \textbf{S2} and \textbf{S3}.
	\end{rem}
	\par Let's motivate those assumptions. Assumption \textbf{S1} is more of a laboratory example to show the strongest regularity result for $\norm{\cdot}^{\mathscr{L}_n}$ we could achieve. 
	Also it generalizes to the non-compact case the hypothesises from Bismut-Bost \cite{BisBost} (cf. Bismut \cite{BisDegQuil}).
	Assumption \textbf{S2} and \textbf{S3} are interesting because the main example of degenerating hyperbolic surfaces (see Construction \ref{const_univ_family}) satisfies them (see Proposition \ref{prop_hyp_metr_assump}). Assumption \textbf{S2} is well-adapted to the curvature theorem, Theorem \ref{thm_curv}, and Assumption \textbf{S3} - to the continuity theorem, Theorem \ref{thm_cont}. 
	\par Now let's argue why instead of more well-known notion of good vector bundles due to Mumford \cite{MumHirz} (see Definition \ref{defn_pre_loglog}), we use the notion of pre-log-log vector bundles due to Burgos Gil-Kramer-K\"uhn \cite{BurKrKun} (see Definition \ref{defn_pre_loglog}).
	By Proposition \ref{prop_hyp_metr_assump}, for our main example of hyperbolic surfaces from Construction \ref{const_univ_family}, the metric $\norm{\cdot}_{X/S}^{{\rm{hyp}}}$ is good. This is stronger than being pre-log-log (see for example \cite[Lemma 4.26]{BurKrKun}). 
	By \cite[\S 4.5]{BurKrKun}, and it follows easily from formulas for degree 0 Bott-Chern forms, the Bott-Chern form associated with good metrics is not necessarily of Poincaré growth (see Definition \ref{defn_loglog_gr}d)).
	 Nevertheless, as it was proved by Burgos Gil-Kramer-K\"uhn \cite[Theorem 4.55]{BurKrKun} (cf. Theorem \ref{thm_bc_preloglog}) in its equivalence class one can choose a representative which is pre-log-log. 
	 Now, the Bott-Chern forms enter naturally in the anomaly formula and the definition of Deligne norms \cite[(6.3.1)]{DeligDet}.
	 Thus, Deligne norm associated with good Hermitian vector bundles is not good in general.
	 Pre-log-log condition is a natural condition to form a class of metrics, such that the associated Deligne norms (see Freixas \cite[Theorem 5.1.3]{FreixTh}) and renormalized Quillen norms (see Theorem \ref{thm_cont}2) almost\footnote{
	 As we say in Theorem \ref{thm_cont}2, the renormalized Quillen norm associated with a pre-log-log Hermitian vector bundle is nice in the sense of Definition \ref{defn_nice}. The notion of niceness is slightly less stronger than the notion of pre-log-log Hermitian vector bundle. Hovewer, as it follows from elliptic regularity (see also the proof of Corollary \ref{corrol_smooth}), a nice vector bundle is pre-log-log if and only if its curvature is smooth over $S \setminus |\Delta|$.
	 } stay in this class.
	This will be used extensively in our forthcoming paper \cite{FinII4} on Deligne-Mumford isometry.
	\par 
	\begin{defn}\label{defn_nice}
		Let $Y$ be a complex manifold and let $D_0$ be a divisor in $Y$.
		\par a) Suppose $D_0$ has normal crossings and the function $f : Y \setminus |D_0| \to \real$ is continuous and has log-log growth along $D_0$ (see Definition \ref{defn_loglog_gr}a)).
		We denote by $[f]_{L^1}$ the current over $Y$, given by the $L^1$-extension of $f$ over $Y$.
		We say that $f$ is \textit{nice with singularities along $D_0$} if the currents $\partial [f]_{L^1}$, $\dbar [f]_{L^1}$, $\partial \dbar [f]_{L^1}$ are defined by the integration against continuous forms over $Y \setminus |D_0|$, which have log-log growth along $D_0$ (see Definition \ref{defn_loglog_gr}b)).
		\par 	b)
		Let $x \in D_0$ and let $U \subset Y$, $x \in U$ be an open subset. Let $h_1, \ldots, h_k$, $k \in \nat$ be local holomorphic functions such that $dh_i(x) \neq 0$, $i = 1, \ldots, k$ and  $n_1, \ldots, n_k \in \nat$ are such that $D_0$ is defined over $U$ by $\{ h_1^{n_1} h_2^{n_2} \cdots h_k^{n_k} = 0 \}$. 
		We say that a smooth function $f : Y \setminus |D_0| \to \real$ is \textit{very nice with singularities along $D_0$} if for any $x \in D_0$ there are smooth functions $f_0, \ldots, f_k : U \to \comp$:
		\begin{equation}
			f = f_0 + \sum_1^{k} f_i |h_i|^2 \ln |h_i|.
 		\end{equation}
 		\par c) Let $L$ be a holomorphic line bundle over $Y$, and let $h^L$ be a continuous Hermitian metric on $L$ over $Y \setminus |D_0|$. For $x \in Y$, fix a local holomorphic frame $\upsilon$ of $L$ in a neighbourhood $U$ of $x$. We say that $h^L$ is \textit{very nice} (resp. \textit{nice})  \textit{with singularities along $D_0$} if for any $x$ and $\upsilon$, the function $\ln h^L(\upsilon, \upsilon)$ is \textit{very nice} (resp. \textit{nice}) \textit{with singularities along $D_0$}.
	\end{defn}
	\begin{rem}\label{rem_nice_chern}
		a) Trivially, for $D_0$ with normal crossings, every \textit{very nice} Hermitian metric with singularities along $D_0$ is \textit{nice} Hermitian metric with singularities along $D_0$. 
		\par b) For a Hermitian metric $h^L$, which is either nice or very nice, we define the first Chern class as a current over $Y$ by
	\begin{equation}\label{eq_c_1_log}
		c_1(L, h^L) := \frac{\partial \dbar [\ln h^L(\upsilon, \upsilon)]_{L^1}}{2 \pi \imun}.
	\end{equation}
	\end{rem}
	
	Our first result describes the singularities of the norm (\ref{rel_quil_norm_intro}) near $|\Delta|$.
	\begingroup
	\setcounter{tmp}{\value{thm}}
	\setcounter{thm}{2} 
	\renewcommand\thethm{\Alph{thm}}
	\begin{sloppypar}
	\begin{thm}[Continuity theorem]\label{thm_cont}
		We consider the line bundle $\mathscr{L}_n$, and the norm $\norm{\cdot}^{\mathscr{L}_n}$ over it.
		\\ \hspace*{0.5cm} \textit{1)} Under Assumption \textbf{S1}, this norm is very nice with singularities along $\Delta$.
		\\ \hspace*{0.5cm} \textit{2)} Under Assumption \textbf{S2}, this norm is nice with singularities along $\Delta$.
		\\ \hspace*{0.5cm} \textit{3)} Under Assumption \textbf{S3}, this norm is continuous over $S$.
	\end{thm}
	\end{sloppypar}
	\endgroup
	\setcounter{thm}{\thetmp}
	\begin{rem}\label{rem_cont}
		a) When $m=0$, Theorem \ref{thm_cont}1 gives the result of Bismut-Bost \cite[Théorème 2.2]{BisBost}. 
		However, we note that the proof presented here relies on  \cite[Théorème 2.2]{BisBost}. In \cite{FinII4}, we give another proof of Theorem \ref{thm_cont}, which relies on the extension of Deligne-Mumford isomorphism and regularity results for Deligne metrics. This would give us, in particular, an alternative proof of \cite[Théorème 2.2]{BisBost}.
		\par b) 
		In the forthcoming paper \cite{FinII3}, under Assumption \textbf{S3}, we exhibit the relation between the restriction of $\mathscr{L}_n$ over the singular fibers and the Quillen norm of the normalisation of the singular fibers.
	\end{rem}
	Theorem \ref{thm_cont} also gives us a possibility to deduce the characterization through the Quillen metric of Grothendick-Knudsen-Mumford determinant line bundle $\det (R^{\bullet} \pi_* (\xi  \otimes \omega_{X/S}(D)^n))^{-1}$ as an extension of Bismut-Gillet-Soulé determinant line bundle $\lambda (j^*(\xi  \otimes \omega_{X/S}(D)^n))$. 
	\begin{sloppypar}
	\begin{cor}\label{cor_ext_character}
		Suppose Assumption \textbf{S2} (resp. \textbf{S3}) hold. Then the line bundle $\mathscr{L}_n$ over $S$, is the only extension of the line bundle
		\begin{equation}\label{eq_line_bundle_bgs}
			\lambda \big(j^*(\xi  \otimes \omega_{X/S}(D)^n) \big)^{12} \otimes \big(\otimes_i \sigma_i^{*}\omega_{X/S} \big)^{-\rk{\xi}} \otimes  \mathscr{O}_S(\Delta)^{\rk{\xi}},
		\end{equation}				
		over $S \setminus |\Delta|$, for which the norm $\norm{\cdot}^{\mathscr{L}_n}$ over $S \setminus |\Delta|$ is \textit{nice} with singularities along $\Delta$ (resp. \textit{continuous} over $S$).
	\end{cor}
	\end{sloppypar}
	We recall that the Chern and Todd forms of the Hermitian vector bundle $(\xi, h^{\xi})$ are defined as
	\begin{equation}\label{eq_ch_td_def}
	\begin{aligned}
		&\ch(\xi, h^{\xi}) = \rk{\xi} + c_1(\xi, h^{\xi}) + \tfrac{1}{2} c_1(\xi, h^{\xi})^2  - c_2(\xi, h^{\xi}) + \cdots,
		\\
		&\td(\xi, h^{\xi}) = \rk{\xi} + \tfrac{1}{2}c_1(\xi, h^{\xi}) + \tfrac{1}{12} \big( c_1(\xi, h^{\xi})^2  + c_2(\xi, h^{\xi}) \big) + \cdots,
	\end{aligned}
	\end{equation}
	where the dots mean higher degree terms.
	\begingroup
	\setcounter{tmp}{\value{thm}}
	\setcounter{thm}{3} 
	\renewcommand\thethm{\Alph{thm}}
	\begin{thm}[Curvature theorem]\label{thm_curv}
		Under Assumption \textbf{S1} (resp. \textbf{S2}), the current 
		\begin{equation}\label{eq_rhs_rrh}
			\pi_*\Big[\td(\omega_{X/S}(D)^{-1}, \norm{\cdot}^{-2}_{X/S}) \ch (\xi, h^{\xi})  \ch(\omega_{X/S}(D)^{n}, \norm{\cdot}_{X/S}^{2n} ) \Big]^{[4]}
		\end{equation}
	is $L^1_{loc}(S)$ (resp. has log-log growth along $\Delta$ in the sense of Definition \ref{defn_loglog_gr}b), so in particular, it is also $L^1_{loc}(S)$). 
	We denote by the same symbol the trivial $L^1$-extension of this current over $S$.  This extension is closed.
	We do the same for the currents $\sigma_i^{*} c_1(\xi, h^{\xi})$, $i = 1,\ldots, m$.
	Then
		\begin{multline}\label{eq_thm_curv}
			\textstyle c_1 \Big( \mathscr{L}_n, (\, \norm{\cdot}^{\mathscr{L}_n} )^2 \Big) 
			=
			 - 12 \pi_*\Big[
			 \td \big(\omega_{X/S}(D)^{-1}, \norm{\cdot}^{-2}_{X/S} \big) 
			 \ch \big( \xi, h^{\xi} \big)  
			 \ch \big( \omega_{X/S}(D)^{n}, \norm{\cdot}_{X/S}^{2n} \big) 
			 \Big]^{[4]}
			\\
			 -
			6
			\sum
			\sigma_i^{*}
			c_1(\xi, h^{\xi}).
		\end{multline}
	\end{thm}
	\endgroup
	\setcounter{thm}{\thetmp}
	\begin{rem}\label{rem_curv}
		a) Assume \textbf{S1} holds and $m = 0$, then Theorem \ref{thm_curv} is exactly the curvature formula of Bismut-Bost \cite[Théorème 2.1]{BisBost}. We note, however, that the proof of Theorem \ref{thm_curv} in this case relies on\cite[Théorème 2.1]{BisBost}. 
		 \par 
		 b) Assume \textbf{S2} holds, then our proof relies on the curvature theorem of Bismut-Gillet-Soulé \cite[Theorem 1.9]{BGS3}, on the proof of Theorem \ref{thm_cont} and on the potential theory of log-log currents, which we develop in Section \ref{sect_pot_th}. In \cite{FinII4}, we give an alternative proof, which uses the extension of Deligne-Mumford isomorphism.
	\end{rem}
	For a complete proof of the following two corollaries, see Section \ref{sect_pf_curv}.
	\begin{cor}\label{cor_cont_pot}
		Assume \textbf{S2} and \textbf{S3} hold. Then the trivial $L^1$-extension of the current (\ref{eq_rhs_rrh}) from Theorem \ref{thm_curv} has a local continuous potential over $S$, which can be written explicitly through a product of the $\norm{\cdot}^{\mathscr{L}_n}$-norm of a holomorphic frame of $\mathscr{L}_n$ and the restriction of $h^{\xi}$-norm of a holomorphic frame of $\xi$ over the images of $\sigma_1, \ldots, \sigma_m$.
	\end{cor}
	\begin{cor}\label{corrol_smooth}
		Assume \textbf{S2} holds, and assume that the current (\ref{eq_rhs_rrh}) is smooth over $S$. Then the Hermitian norm $\norm{\cdot}^{\mathscr{L}_n}$ on the line bundle $\mathscr{L}_n$ is smooth over $S$.
	\end{cor}
	\begin{rem}\label{rem_cor_smooth}
	\begin{sloppypar}
		Quite easily, if $h^{\xi}$ and $\norm{\cdot}_{X/S}$ satisfy \textbf{S1}, and the forms $c_1(\xi, h^{\xi})$, $c_2(\xi, h^{\xi})$, $c_1(\omega_{X/S}(D), \norm{\cdot}_{X/S})$ vanish in the neighbourhood of $\Sigma_{X/S}$, then since $\pi|_{X \setminus \Sigma_{X/S}}$ is a submersion, the current (\ref{eq_rhs_rrh}) is smooth over $S$, so the hypothesises of Corollary \ref{corrol_smooth}a) are satisfied.
	\end{sloppypar}
	\end{rem}
	\par Now let's describe some of the applications of those results to the moduli space $\modul_{g, m}$ of $m$-pointed stable curves of genus $g$, $2g - 2 + m > 0$. We denote by $\modulcomp_{g, m}$ the \textit{Deligne-Mumford compactification} of $\modul_{g, m}$, by $\partial \modul_{g, m} := \modulcomp_{g, m} \setminus \modul_{g, m}$ the \textit{compactifying divisor}, by $\univcurv_{g, m}$ and $\univcurvcomp_{g, m}$ the universal curves over $\modul_{g, m}$ and $\modulcomp_{g, m}$ respectively. 
	We denote by $\Omega : \univcurvcomp_{g, m} \to \modulcomp_{g, m}$ the universal projection. We denote by $D_{g, m}$ the divisor on $\univcurvcomp_{g, m}$, formed by $m$ fixed points. We denote by $\omega_{g, m}$ the relative canonical line bundle of $\Omega$, by $\otimes \sigma_i^{*} \omega_{g, m}$ the determinant of the restriction of $\omega_{g, m}$ to the divisor $D_{g, m}$, and by $\omega_{g, m}(D)$ the twisted relative canonical line bundle:
	\begin{equation}\label{def_rel_can_modul}
		\omega_{g, m}(D) := \omega_{g, m} \otimes \mathscr{O}_{\univcurvcomp_{g, m}}(D_{g, m}).
	\end{equation}
	\par By the uniformization theorem  (cf. \cite[Chapter IV]{FarKra}, \cite[Lemma 6.2]{Auvr}, \cite{AuvrMaArx}), we endow $\omega_{g, m}(D)$ with the Hermitian norm $\norm{\cdot}_{g, m}^{\rm{hyp}}$, such that its restriction over each fiber of $\pi$ induces by Construction \ref{const_norm_div} the canonical hyperbolic metric of constant scalar curvature $-1$. This endows the determinant line bundle $\lambda(j^*(\omega_{g, m}(D)^n))$, $n \leq 0$, which is also sometimes called the \textit{Hodge line bundle}, with the induced Quillen metric $\norm{\cdot}^{Q, n}_{g, m}$.
	Also the line bundle $\otimes \sigma_i^{*} \omega_{g, m}$ is endowed with the associated Wolpert norm $\norm{\cdot}^W_{g, m}$.
	We denote by $\omega_{WP}$ the Weil-Petersson form over $\modul_{g, m}$ (cf. Section \ref{sect_orb_str}).
	The following corollaries will be proved in Section \ref{sect_pinch}.
	\begin{cor}\label{cor_cont_hodge}
		The norm 
		\begin{equation}\label{eq_renorm_hodge_norm}
			\norm{\cdot}^{H, n}_{g, m} :=
			(\, \norm{\cdot}^{Q, n}_{g, m})^{12} 
			\otimes
			(\, \norm{\cdot}^W_{g, m})^{-1} 
			\otimes 
			\, \norm{\cdot}_{\partial \modul_{g,m}}^{\rm{div}}
		\end{equation}
		on the line bundle
		\begin{equation}\label{eq_renorm_hodge}
			\lambda_{g, m}^{H, n} :=
			\lambda(j^*(\omega_{g, m}(D)^n))^{12} 
			\otimes 
			(\otimes \sigma_i^{*} \omega_{g, m})^{-1} 
			\otimes 
			\mathscr{O}_{\modulcomp_{g,m}}(\partial \modul_{g,m})
		\end{equation}
		is nice with singularities along $\partial \modul_{g, m}$. Moreover, it extends continuously over $\modulcomp_{g, m}$ and it is smooth over $\modul_{g, m}$.
	\end{cor}
	By Corollary \ref{cor_cont_hodge} and Remark \ref{rem_nice_chern}, we see that the first Chern form of $(\lambda_{g, m}^{H, n}, (\, \norm{\cdot}^{H, n}_{g, m})^{2})$ is well-defined as a current over $\modulcomp_{g, m}$.
	Let's state the curvature theorem in this context.
	\begin{cor}\label{cor_form_TZ_type}
		The form $\omega_{WP}$ has log-log growth along the boundary $\partial \mathscr{M}_{g, m}$ of the moduli space of curves. We denote by $[\omega_{WP}]_{L^1}$ its $L^1$-extension to $\modulcomp_{g, m}$.
		It is closed and the following identity holds for currents over $\modulcomp_{g, m}$
		\begin{equation}
			c_1 
			\Big(
				\lambda_{g, m}^{H, n}
				,
				(\, \norm{\cdot}^{H, n}_{g, m})^2
			\Big)
			=
			- \pi^{-2} \Big( 6n^2 - 6n + 1\Big) \big[\omega_{WP}\big]_{L^1}.
		\end{equation}
	\end{cor}
	\begin{rem}
		a) 	
		By the results of Wolpert \cite[Theorem 5]{Wol07}, \cite[Corollary 5.11]{WolpChForm86}, Theorem \ref{thm_curv} gives the extension of the curvature theorem of Takhatajan-Zograf \cite[Theorem 1]{TakZog} from $\modul_{g, m}$ to $\modulcomp_{g, m}$. Our methods are very different from the methods of Takhatajan-Zograf, as we don't use the  variational approach with Beltrami differentials. 
		\\ \hspace*{0.5cm} b) 
		The fact that the Weil-Petersson form has log-log growth along the boundary $\partial \mathscr{M}_{g, m}$ also follows from an old result of Masur \cite[Theorem 1]{MasurWP}. See also the recent article of Melrose-Zhu \cite{MelZhu} for related results.
	\end{rem}
	Let's state the following corollaries, regarding the Weil-Petersson form. 
	\begin{cor}\label{cor_wp_cont_pot}
		The Weil-Petersson form $\omega_{WP}$ has a local continuous potential.
	\end{cor}
	\begin{rem}
		Corollary \ref{cor_wp_cont_pot} was originally proved by Wolpert in \cite[\S 2]{WolpPos85}.
		He later used it to give a complex-analytic proof of the ampleness of Weil-Petersson form.  
		Our method of the proof is constructive, and doesn't use $\partial \overline{\partial}$-lemma, thus, it is very different from the non-constructive proof in \cite[\S 2]{WolpPos85}. 
	\end{rem}
	\begin{cor}\label{cor_wp_vol}
		We can decompose the Weil-Petersson form $\omega_{WP}$ as 
		\begin{equation}\label{eq_wolp_decomp}
			\omega_{WP} = - \pi^2 \alpha + d \beta,
		\end{equation}
		for some smooth forms $\alpha, \beta$ over $\modul_{g, m}$.
		Moreover, there is a smooth Hermitian metric $h_{sm}$ on $\lambda_{g, m}^{H, 0}$ over $\modulcomp_{g, m}$ such that
		\begin{equation}
			\alpha = c_1(\lambda_{g, m}^{H, 0}, h_{sm}),
		\end{equation}
		and $\beta$, $d \beta$ have log-log growth along $\partial \modulcomp_{g, m}$.
		In particular, we have
		\begin{equation}\label{eq_wolp_volume}
			\int_{\modul_{g, m}}\omega_{WP}^{\wedge(3g - 3 + m)} 
			= 
			(-\pi^2)^{3g - 3 + m} \int_{\modulcomp_{g, m}} c_1(\lambda_{g, m}^{H, 0})^{\wedge(3g - 3 + m)} .
		\end{equation}
		As a consequence, we see that the left-hand side of (\ref{eq_wolp_volume}) is a rational multiple of a power of $\pi$.
	\end{cor}
	\begin{rem}
		a) The decomposition (\ref{eq_wolp_decomp}) can be also deduced from studying the singularities of the Deligne metric on the Deligne-Weil product, as it was done implicitly by Freixas in his PhD thesis \cite[Theorem 5.1.3]{FreixTh}.
		\par b) The identity (\ref{eq_wolp_volume}) was originally proved by Wolpert. In \cite{WolpPos85} he showed that $\omega_{WP}$ extends smoothly to the boundary $\partial \modul_{g, m}$ with respect to the Fenchel-Nielsen coordinates, and in \cite{WolpVolCoh} he showed that the De Rham cohomology class $[\pi^{-2} \omega_{WP}]_{DR}$ of $\pi^{-2} \omega_{WP}$ lies in $H^2(\modulcomp_{g, m}, \mathbb{Q})$, and it is equal to $c_1(\lambda_{g, m}^{H, 0})$. Since $\omega_{WP}$ is originally defined in Bers coordinates, defining the complex structure on $\modulcomp_{g, m}$, our proof of (\ref{eq_wolp_volume}) is different.
	\end{rem}
	Finally, let's describe our last application. For this, let's recall that Deligne in \cite[\S 7]{DeligDet} defined a holomorphic line bundle $\langle \omega_{g, m}(D), \omega_{g, m}(D) \rangle$ over $\modulcomp_{g,m}$, which is now called the Deligne-Weil product. For $m = 0$, in \cite[\S 8]{DeligDet}, he endowed it with the Hermitian norm $\norm{\cdot}^{\rm{Del}}_{g, m}$, which is now called the Deligne norm. Later, Freixas in his PhD thesis \cite[Theorem 5.1.3]{FreixTh} generalized the construction of this norm for $m \in \nat$. The natural isomorphism
	\begin{equation}\label{eq_del_isom}
		\lambda_{g, m}^{H, n} \to \big \langle \omega_{g, m}(D), \omega_{g, m}(D) \big \rangle^{-(6 n^2 - 6n + 1)}
	\end{equation}
	was constructed by Deligne in \cite[Théorème 9.9]{DeligDet} for $m = 0$. Then Freixas in \cite[Theorem 3.10]{FreixasARR} extended it for $m \in \nat$. Those isomorphisms are canonical and can be characterized uniquely up to a multiplication by $-1$ as the morphisms which respect $\integ$-structure of the corresponding line bundles, see \cite{Knudsen1976}, \cite{Knud2}, \cite{DeligDet}, \cite{FreixasARR}.
	\begin{cor}\label{cor_del_iso}
		The isomorphism (\ref{eq_del_isom}) is an isometry up to a constant when the left-hand-side is endowed with the norm (\ref{eq_renorm_hodge_norm}) and the right-hand side is endowed with the norm induced by $\norm{\cdot}^{\rm{Del}}_{g, m}$.
	\end{cor}
	\begin{rem}
		For $m = 0$, this was proved by Deligne \cite[Théorème 11.4]{DeligDet} and Gillet-Soulé in \cite[Proposition 1.5.2]{GilSoulTodd}, \cite{GilSoul92}. For $m \in \nat$, $n = 0$, this was proved by Freixas in \cite[Theorem 6.1]{FreixasARR}. We note that Gillet-Soulé and Freixas also explicitly computed the constant up to which this isometry holds. In \cite{FinII4}, we extend their result for all families of curves with cusps.
	\end{rem}
	\begin{sloppypar}	
	 	Finally, let's mention that in the related work of Albin-Rochon \cite{AlbRoch}, authors obtain local family index formula for the direct image $R^{\bullet} \pi_* (\omega_{g, m}(D)^n)$, which is a holomorphic vector bundle, over $\mathscr{M}_{g,m}$. 
	 	In our situation, the cohomology of the fiber may vary, and similarly to \cite{BGS1}, \cite{BGS2}, \cite{BGS3}, we work only with the first Chern form.
	\end{sloppypar}	
	\par Now let's describe the plan of this article. 
	In Section 2, we recall the basic definitions of the subject, we recall the definition of the Quillen metric, different notions of singularities of vector bundles and Bott-Chern forms. We also see in Section \ref{sect_sing_metr} how Corollary \ref{cor_ext_character} follows from Theorem \ref{thm_cont}.
	In Section 3, we prove an analytic proposition, which studies the singularities of a push-forward of a differential form in f.s.o. Then we use it to prove Theorem \ref{thm_cont}. 
	In Section 4, we develop potential theory for currents of log-log growth and then we prove Theorem \ref{thm_curv} and Corollaries \ref{cor_cont_pot}, \ref{corrol_smooth}.
	In Section 5, we recall the necessary prerequisites related to the moduli space of pointed stable curves and prove Corollaries \ref{cor_cont_hodge}, \ref{cor_form_TZ_type}, \ref{cor_wp_cont_pot}, \ref{cor_wp_vol}, \ref{cor_del_iso}.
	\par \textbf{Notation.} For $\epsilon > 0$, we denote
\begin{equation}
	D(\epsilon) = \{ z \in \comp : |z| < \epsilon \}, \qquad 	D^*(\epsilon) = \{ z \in \comp : 0 <  |z| < \epsilon \}.
\end{equation}
For a vector space $E$, we denote $\det E := \Lambda^{\max} E$. 
\par For a holomorphic vector bundle $\xi$ over a complex manifold $X$ with a Hermitian metric $h^{\xi}$ over $X$, the pair $(\xi, h^{\xi})$ is called a \textit{Hermitian vector bundle} over $X$. 	
\par For a divisor $D_0 \subset Y$ in a complex manifold $Y$, we denote by $s_{D_0}$ the canonical holomorphic section of $\mathscr{O}_Y(D_0)$, ${\rm{div}}(s_{D_0}) = D_0$, and by $\delta_{D_0}$ the current of integration along $D_0$.
For a current $T$ over $Y \setminus D_0$, which is in $L^1_{loc}(Y)$, we denote by $[T]_{L^1}$ the $L^1$ extension of $T$ over $Y$.
\par Let $\alpha=(\alpha_1, \ldots, \alpha_q) \in \nat^{q}$, $q \in \nat$ be a multi-index. We denote by $|\alpha| = \sum \alpha_i$.
	\par {\bf{Acknowledgements.}} This work is a part of our PhD thesis, which was done at Université Paris Diderot. We would like to express our deep gratitude to our PhD advisor Xiaonan Ma for his teaching, overall guidance, constant support and invaluable comments on the preliminary version of this article.

\section{Families of nodal curves and related notions}
	In this section we recall the relevant notations.
	More precisely, in Section 2.1, we recall the notion of the analytic torsion from the first paper of these series, \cite{FinII1} (for related notions of analytic torsions, see Lundelius \cite{Lun93}, Jorgenson-Lundelius \cite{JorLund}, \cite{JorLundMain}; Albin-Rochon \cite{AlbRoch_gen}, \cite{AlbRoch}; see also \cite[\S 2.2]{FinII1} for a brief summary about the connections between those definitions).
	In Section 2.2, we recall the definition of holomorphic families of Riemann surfaces with ordinary singularities, we define the notion of a family of surfaces with cusps and the Wolpert norm on the restriction of the relative canonical line bundle to the cusps.
	In Section 2.3, we recall the properties of the determinant line bundle due to Grothendick-Knudsen-Mumford \cite{Knudsen1976} and Bismut-Gillet-Soulé \cite{BGS3}. We also recall the definition of the Quillen norm for non-compact surfaces, which was done in \cite{FinII1}.
	In Section 2.4, we recall several notions of singularities of Hermitian metrics on holomorphic line bundles.
	Finally, in Section 2.5, we recall the theory of Bott-Chern currents for singular Hermitian metrics and prove some useful properties of those currents.

\subsection{The analytic torsion}\label{sect_recall_relantors}
	Let $\overline{M}$ be a compact Riemann surface, and let $D_M = \{ P_1^{M}, \ldots, P_m^{M} \}$ be a finite set of distinct points in $\overline{M}$. Let $g^{TM}$ be a Kähler metric on the punctured Riemann surface $M = \overline{M} \setminus D_M$. 
	\par For $\epsilon \in ]0,1[$, $i = 1, \ldots, m$, let $z_i^{M} : \overline{M} \supset V_i^M(\epsilon) \to D(\epsilon) = \{ z \in \comp : |z| \leq \epsilon \}$ be a local holomorphic coordinate around $P_i^{M}$, and 
	\begin{equation}\label{defn_v_i}
		V_i^{M}(\epsilon) := \{x \in M :  |z_i^{M}(x)| < \epsilon \}.
	\end{equation}	 
	We say that $g^{TM}$ is \textit{Poincaré-compatible} with coordinates $z_1^{M}, \ldots, z_m^{M}$ if for any $i = 1, \ldots, m$,  there is $\epsilon > 0$ such that $g^{TM}|_{V_i^{M}(\epsilon)}$ is induced by the Hermitain form
	\begin{equation}\label{reqr_poincare}
		\frac{\imun dz_i^{M} d\overline{z}_i^{M}}{ \big| z_i^{M}  \ln |z_i^{M}| \big|^2}.
	\end{equation}
	We say that $g^{TM}$ is a \textit{metric with cusps} if it is Poincaré-compatible with some holomorphic coordinates of $D_M$.
	A triple $(\overline{M}, D_M, g^{TM})$ of a Riemann surface $\overline{M}$, a set of punctures $D_M$ and a metric with cusps $g^{TM}$ is called a \textit{surface with cusps} (cf. \cite{MullerCusp}). 
	\par From now on, we fix a surface with cusps $(\overline{M}, D_M, g^{TM})$ and a Hermitian vector bundle $(\xi, h^{\xi})$ over it.  
	We denote by $\omega_{M} := T^{*(1,0)}\overline{M}$ the \textit{canonical line bundle} over $\overline{M}$. 
	We denote by $\norm{\cdot}_{M}^{\omega}$ the norm induced on $\omega_{\overline{M}}$ by $g^{TM}$ over $M$. 
	Let $\mathscr{O}_M(D_M)$ be the line bundle associated with the divisor $D_M$. 	
	The \textit{twisted canonical line bundle} is defined as 
	\begin{equation}
			\omega_M(D) :=  \omega_{\overline{M}} \otimes  \mathscr{O}_{\overline{M}}(D_M).
	\end{equation}		
	The metric $g^{TM}$ endows by Construction \ref{const_norm_div} the line bundle $\omega_M(D)$ with the induced Hermitian metric $\norm{\cdot}_{M}$ over $M$.
	\begin{sloppypar}
	We denote by $\laplcomp^{\xi \otimes \omega_M(D)^n}$ the Kodaira Laplacian associated with $g^{TM}$ and $(\xi \otimes \omega_M(D)^n, h^{\xi} \otimes \, \norm{\cdot}_M^{2n})$.
	In this article we only consider the action of $\laplcomp^{\xi \otimes \omega_M(D)^n}$ on the sections of degree $0$.
	\par 
	We recall that for $m=0$, the analytic torsion was defined by Ray-Singer \cite{Ray73} as the regularized determinant of $\laplcomp^{\xi \otimes \omega_M(D)^n}$. 
	More precisely, let $\lambda_i, i \in \nat$ be the non-decreasing sequence of non-zero eigenvalues of $\laplcomp^{\xi \otimes \omega_M(D)^n}$. 
	Classically, the associated zeta-function 
	\begin{equation}\label{defn_zeta_comp}
		\zeta_M(s) := \sum \lambda_i^{s}
	\end{equation}
	 is defined for $\Re (s) > 1$, $s \in \comp$ and it extends meromorphically to the entire $s$-plane. This extension is holomorphic at $0$, and the \textit{analytic torsion} is defined by
	\end{sloppypar}
	\begin{equation}\label{defn_an_t_st}
		T(g^{TM}, h^{\xi} \otimes \norm{\cdot}_M^{2n}) := \exp(- \zeta'_M(0)).
	\end{equation}
	However, for $m > 0$, the heat operator associated with $\laplcomp^{\xi \otimes \omega_M(D)^n}$ is no longer of trace class. Thus, the definition (\ref{defn_an_t_st}) is no longer applicable, but, as it was shown in the previous paper of this series \cite[\S 2.2]{FinII1}, by taking out the diverging part in the definition of the heat trace, we can extend the definition of zeta-function $\zeta_M(s)$ to define the \textit{analytic torsion} $T(g^{TM}, h^{\xi} \otimes \norm{\cdot}_M^{2n})$ for $m \geq 0$ by the same formula (\ref{defn_an_t_st}).

\subsection{Families of nodal curves}\label{sect_fnc}
In this section we recall the definition of a holomorphic family of Riemann surfaces with ordinary singularities (cf. Bismut-Bost \cite{BisBost}) and some of its properties. We introduce its cusped version, which we call a family of surfaces with cusps, and we also define the Wolpert norm.

\begin{defn}
	A \textit{holomorphic family of Riemann surfaces with ordinary singularities} is a holomorphic, proper, surjective map $\pi: X \to S$ of complex manifolds, such that for every $t \in S$, the space $X_t := \pi^{-1}(t)$ is a complex curve whose singularities are at most ordinary double points. As a shortcut, we will call $\pi$ a f.s.o. (from french “famille à singularités ordinaires").
\end{defn}

\begin{prop}[{\cite[Proposition 3.1]{BisBost}}]\label{prop_coord}
	Let $\pi : X \to S$ be a f.s.o. 
	Then for every $x \in X$, there are local holomorphic coordinates $(z_0, \ldots, z_q)$ of $x \in X$ and $(w_1, \ldots, w_q)$ of $\pi(x) \in S$, such that $\pi$ is locally defined by one of the following identities
	\begin{align}
		& w_i = z_i,  &&\text{for} \quad i=1, \ldots, q,  \label{eq_pr_nonsing}
		\\
		& w_1 = z_0 z_1; \quad w_i = z_i,  &&\text{for} \quad i=2, \ldots, q. \label{eq_pr_sing}
	\end{align}
\end{prop}
 
\begin{cor}[{\cite[\S 3(a)]{BisBost}}]\label{cor_sigma}
	Let $\pi : X \to S$ be a f.s.o., and let $\Sigma_{X/S} \subset X$ be the locus of double points of the fibers of $\pi$. Then the following holds
	\\ \hspace*{0.5cm} 	a) $\Sigma_{X/S}$ is a submanifold of $X$ of codimension $2$;
	\\ \hspace*{0.5cm} 	b) the map $\pi|_{\Sigma_{X/S}} : \Sigma_{X/S} \to S$ is a closed immersion;
	\\	 \hspace*{0.5cm}	c) the map $\pi|_{X \setminus \Sigma_{X/S}} : X \setminus \Sigma_{X/S} \to S$ is a submersion.
	\\
	In particular, the direct image $\Delta = \pi_*(\Sigma_{X/S})$ is a divisor in $S$.
\end{cor}
\begin{notat}
	We use the notation $\Delta$, $\Sigma_{X/S}$ for the divisor and the submanifold from Corrolary \ref{cor_sigma}.
\end{notat}

For a complex manifold $X$, we denote by $\Omega_X$ the sheaf of holomorphic sections of the vector bundle $T^{*(1,0)}X$, and by $\omega_X$ the line bundle $\det (T^{*(1,0)}X)$. 
\par Let's recall the construction of the \textit{relative canonical} line bundle $\omega_{X/S}$ of a f.s.o. $\pi : X \to S$.  Define the sheaf $\Omega_{X/S}$ by the exact sequence:
\begin{equation}\label{ex_seq_1}
\pi^* \Omega_S \to \Omega_X \to \Omega_{X/S} \to 0.
\end{equation}
By Corollary \ref{cor_sigma}, the exact sequence (\ref{ex_seq_1}) becomes exact to the left when restricted to $X \setminus \Sigma_{X/S}$:
\begin{equation}\label{ex_seq_2}
0 \to \pi^* \Omega_S|_{X \setminus \Sigma_{X/S}} \to \Omega_X|_{X \setminus \Sigma_{X/S}} \to \Omega_{X/S}|_{X \setminus \Sigma_{X/S}} \to 0.
\end{equation}
By taking determinants of (\ref{ex_seq_2}), we deduce the isomorphism
\begin{equation}
	 \Omega_{X/S}|_{X \setminus \Sigma_{X/S}} = (\omega_X \otimes \pi^* \omega_S^{-1})|_{X \setminus \Sigma_{X/S}}.
\end{equation}
We define
\begin{equation}
	\omega_{X/S} := \omega_X \otimes \pi^* \omega_S^{-1}.
\end{equation}
Then $\omega_{X/S}$ is the unique extension of $\Omega_{X/S}|_{X \setminus \Sigma_{X/S}}$ over $X$. This line bundle is called the \textit{relative canonical} line bundle of $\pi : X \to S$.
\par Let $x \in \Sigma_{X/S}$. Take local coordinates $(z_0, \ldots , z_q)$ on an open neighbourhood $V$ of $x \in X$ and local coordinates $(w_1, \ldots, w_q)$ of $\pi(x) \in S$, as in (\ref{eq_pr_sing}).
Then the manifold $\Sigma_{X/S} \cap V$ is given by
\begin{equation}
	\{ z_0 = 0 \text{ and }  z_1 = 0\}.
\end{equation}
Consider the sections $dz_0/z_0$ and $dz_1/z_1$ of $\Omega_X$, defined over the sets $\{ z_0 \neq 0 \}$ and $\{ z_1 \neq 0 \}$ respectively. The images of $dz_0/z_0$ and $-dz_1/z_1$ in $\omega_{X / S}$ coincide over $\{z_0z_1 \neq 0\}$, since
\begin{equation}\label{eq_section_omegarel}
	\frac{dz_0}{z_0} + \frac{dz_1}{z_1} = \pi^* \frac{dw_1}{w_1}.
\end{equation}
Thus, they define a nowhere vanishing section $\sigma$ of $\omega_{X / S}$ over $V \setminus \Sigma_{X/S}$. Since $\Sigma_{X/S}$ is of codimension $2$, the section $\sigma$ extends to a nowhere vanishing section over $V$ of the line bundle $\omega_{X/S}$.

\begin{defn}[Family of surfaces with cusps]\label{defn_fsc}
	A \textit{holomorphic family of Riemann surfaces with ordinary singularities and cusps} is a f.s.o. $\pi : X \to S$, disjoint sections $\sigma_1, \ldots, \sigma_m : S \to X \setminus \Sigma_{X/S}$ and a Hermitian metric $\norm{\cdot}^{\omega}_{X/S}$ on $\omega_{X / S}$ over $\pi^{-1}(S \setminus |\Delta|) \setminus ( \cup_i \Im(\sigma_i) )$, such that for any $t \in S \setminus |\Delta|$, the restriction of $\, \norm{\cdot}^{\omega}_{X/S}$ over $\pi^{-1}(t) \setminus (\cup_i \sigma_i(t))$ induces the Kähler metric $g^{TX_t}$ over $X_t \setminus (\cup_i \sigma_i(t))$ such that the associated triple $(X_t, \{ \sigma_1(t), \ldots, \sigma_m(t) \}, g^{TX_t})$ becomes a surface with cusps. As a short-cut, we call $(\pi; \sigma_1, \ldots, \sigma_m; \norm{\cdot}^{\omega}_{X/S})$ a f.s.c.
\end{defn}

\begin{notat}
	From now on, we denote by $D_{X/S}$ the divisor on $X$, given by $\Im (\sigma_1) + \ldots + \Im( \sigma_m)$. We denote by $D_{X_t}$ its restriction on $X_t := \pi^{-1}(t)$, $t \in S$.
\end{notat}

\begin{defn}[Wolpert norm]\label{defn_wolp}
	Let $(\pi; \sigma_1, \ldots, \sigma_m; \norm{\cdot}^{\omega}_{X/S})$ be a f.s.c.
	Let $t \in S \setminus |\Delta|$, and let $z_i$ be a holomorphic coordinate of $\sigma_i(t) \subset X_t$ such that (see (\ref{reqr_poincare}))
	\begin{equation}
		\norm{dz_i}^{\omega}_{X/S} = |z_i \ln |z_i||.
	\end{equation}
	By the uniformization theorem, such a holomorphic coordinate is uniquely defined up to a multiplication by a  unitary complex constant. We define the norm $\norm{\cdot}^W_{i}$ on $\sigma_i^{*}(\omega_{X / S})$ pointwise by
	\begin{equation}
		\norm{\sigma_i^{*} dz_i}^W_{i}(t) = 1.
	\end{equation}
	The \textit{Wolpert norm} $\norm{\cdot}^W_{X/S}$ is defined as the product norm on $\otimes_i \sigma_i^{*}(\omega_{X / S})$, induced by $\norm{\cdot}^W_{i}$.
\end{defn}
	In general, we make no claims about the smoothness of $\norm{\cdot}_{X/S}^{W}$, but we hope to come back to this question soon.
\begin{rem}\label{rem_wol_cont}
	Let $(\pi; \sigma_1, \ldots, \sigma_m; \norm{\cdot}^{\omega, {\rm{hyp}}}_{X/S})$ be from Construction \ref{const_univ_family}. Wolpert in \cite[Definition 1]{Wol07} defined the norm $\norm{\cdot}^{W, {\rm{hyp}}}_{X/S}$, which coincides with the norm from Definition \ref{defn_wolp} in this particular case. In \cite[Theorem 5]{Wol07} he showed that  $\norm{\cdot}^{W, {\rm{hyp}}}_{X/S}$ is smooth over $S \setminus |\Delta|$.
\end{rem}

\subsection{Determinant line bundles and Quillen norms}
	In this section we recall the notion of the determinant line bundle due to Grothendick-Knudsen-Mumford \cite{Knudsen1976} and then, similarly to Bismut-Gillet-Soulé \cite{BGS3}, but basing on the definition of the analytic torsion from \cite[Definition 2.17]{FinII1}, we introduce the notion of the Quillen norm on the determinant line bundle.
	\par 
	Let $\pi : X \to S$ be a f.s.o., and let $\xi$ be a holomorphic vector bundle over $X$. We denote
	\begin{equation}
		\det (R^{\bullet} \pi_* \xi)_t := \det H^{0}(X_t, \xi) \otimes (\det H^{1}(X_t, \xi))^{-1}, \quad t \in S,
	\end{equation}
	where we identified $\xi$ with its sheaf of holomorphic sections.
	By Grothendick-Knudsen-Mumford \cite{Knudsen1976} (cf. \cite[Proposition 4.1]{BisBost}) the family of complex lines $(\det (R^{\bullet} \pi_* \xi)_t )_{t \in S}$ is endowed with a natural structure of holomorphic line bundle $\det (R^{\bullet} \pi_* \xi)$ over $S$. 
	\par Now, suppose $(\pi; \sigma_1, \ldots, \sigma_m; \norm{\cdot}^{\omega}_{X/S})$ is a f.s.c.
	For $t \in S \setminus |\Delta|$, we denote by $d \vol_{X_t}$ the Riemannian volume form on $X_t \setminus (\cup \sigma_i(t))$, induced by $\norm{\cdot}_{X/S}^{\omega}$ on the fiber $X_t$.
	Endow the twisted canonical line bundle $\omega_{X/S}(D)$ with the norm $\norm{\cdot}_{X/S}$ from Construction \ref{const_norm_div}.	
	Let $\xi$ be endowed with a Hermitian metric $h^{\xi}$ over $\pi^{-1}(S \setminus |\Delta|)$.
	For $n \in \integ, n \leq 0$, we define the $L^2$-scalar product $\langle \cdot, \cdot \rangle_{L^2}$ on $\ccal^{\infty}(X_t, \xi \otimes \omega_{X/S}(D)^n)$ and $\ccal^{\infty}(X_t, \xi \otimes \omega_{X/S}(D)^n \otimes \overline{\omega}_{X/S})$ by
	\begin{equation}\label{defn_l2_scal}
		\scal{s_1}{s_2}_{L^2} = \int_{X_t} \scal{s_1(x)}{s_2(x)}_{h} d \vol_{X_t}(x),
	\end{equation}
	where $s_1, s_2$ are either in 
	$\ccal^{\infty}(X_t, \xi \otimes \omega_{X/S}(D)^n)$ or in 
	$\ccal^{\infty}(X_t, (\xi \otimes \omega_{X/S}(D)^n)^* \otimes \overline{\omega}_{X/S})$,
	and $\langle \cdot, \cdot \rangle_{h}$ is the pointwise Hermitian product induced by $h^{\xi}, \norm{\cdot}^{\omega}_{X/S}$ and $\norm{\cdot}_{X/S}$. 
	As we explained in \cite[Section 2.1]{FinII1}, the right-hand side of (\ref{defn_l2_scal}) is finite, and (\ref{defn_l2_scal}) defines the $L^2$-scalar product on the vector spaces $H^{0}(X_t, \xi \otimes \omega_{X/S}(D)^n)$, $H^{1}(X_t, \xi \otimes \omega_{X/S}(D)^n)$.
	We denote by $\norm{\cdot}_{L^2}(g^{TX_t}, h^{\xi} \otimes \norm{\cdot}_{X/S}^{2n})$ the induced $L^2$-norm on the complex line $(\det (R^{\bullet} \pi_* (\xi  \otimes \omega_{X/S}(D)^n))_t)^{-1}$. 
	\begin{sloppypar}
	\begin{defn}\label{defn_quil_norm}
		The Quillen norm on the line bundle $(\det (R^{\bullet} \pi_* (\xi  \otimes \omega_{X/S}(D)^n)))^{-1}$ over $S \setminus |\Delta|$ is defined for $t \in S \setminus |\Delta|$ by 
		\begin{equation}
			\norm{\cdot}_{Q} \big( g^{TX_t}, h^{\xi} \otimes \, \norm{\cdot}_{X/S}^{2n} \big)
			:= 
			 T \big( g^{TX_t}, h^{\xi} \otimes \, \norm{\cdot}_{X/S}^{2n} \big)^{1/2} \cdot
			\norm{\cdot}_{L^2} \big( g^{TX_t}, h^{\xi} \otimes \norm{\cdot}_{X/S}^{2n} \big).
		\end{equation}
	\end{defn}
	\end{sloppypar}

\subsection{Singular Hermitian vector bundles}\label{sect_sing_metr}
	In this section we recall several notions of singularities for Hermitian vector bundles. Then we show how Theorem \ref{thm_cont} implies Corollary \ref{cor_ext_character}.
	\par
	We work with a complex manifold $Y$ of dimension $q+1$, a \textit{normal crossing divisor} $D_0 \subset Y$ and a submanifold $F \subset Y$.
	\begin{defn}\label{defn_adapt_chart}
		A triple $(U; z_0, \ldots, z_q; l)$ of an open set $U \subset Y$, coordinates $z_0, \ldots, z_q : U \to \comp$ and $l \in \nat$ is called an \textit{adapted chart} for $D_0$ (resp. $F$) at $x \in |D_0|$ (resp. $x \in F$) if $U = \{ (z_0, \ldots, z_q) \in \comp^{q+1} : |z_i| < 1, \text{ for all } i = 0, \ldots, q \}$ and $|D_0| \cap U$ (resp. $F \cap U$) is defined by $\{z_0 \cdots z_l = 0\}$ (resp. $\{z_0 = 0, \ldots, z_l = 0\}$).
	\end{defn}
	\begin{notat}
		Let $(U; z_0, \ldots, z_q; l)$ be an \textit{adapted chart} for $D_0$. We denote
		\begin{equation}
			d \zeta_k = 
			\begin{cases}
				\hfill dz_k/(z_k \ln |z_k|), & \text{if} \quad 0 \leq k \leq l, \\
				\hfill dz_k, & \text{if} \quad l+1 \leq k \leq q.
			\end{cases}
		\end{equation}
	\end{notat}
	\begin{defn}\label{defn_loglog_gr}
		a) \cite[Definition 2.17]{BurKrKun} A differential form over $Y \setminus |D_0|$ (resp. a locally bounded section of the wedge algebra on cotangent space) has \textit{log-log growth of order $k \in \nat$ (resp. weakly log-log growth) on $Y$, with singularities along $D_0$}, if it can be expressed as a linear combination of monomials constructed using $d \zeta_k, \overline{d \zeta_k}$, $k = 0, \ldots, q$ with coefficients $f \in \ccal^{\infty}(Y \setminus |D_0|)$ (resp. $f \in L^{\infty}(Y \setminus |D_0|)$) such that for any $\textbf{k}, \textbf{k}' \in \nat^{q+1}$, $|\textbf{k}| + |\textbf{k}'| \leq K$ (resp. for $\textbf{k}, \textbf{k}' = 0$), for some adapted chart $(U; z_0, \ldots, z_q; l)$ of $D_0$ at $x \in |D_0|$, and for some $C > 0$, $p \in \nat$, we have
		\begin{equation}\label{eq_loglog_grwth_cond}
			\Big|
			\frac{\partial^{|\textbf{k}|}}{\partial z^{\textbf{k}}} 
			\frac{\partial^{|\textbf{k}'|}}{\partial \overline{z}^{\textbf{k}'}}
			f(z_0, \ldots, z_q)
			\Big| 
			\leq 
			C |z|^{- |\textbf{k}_0| - |\textbf{k}'_0|} \prod_{k=0}^{l} \big( \ln | \ln |z_k| |\big)^p   + C,
		\end{equation}
		where $\textbf{k}_0, \textbf{k}'_0 \in \nat^{l+1}$ are the projections of $\textbf{k}, \textbf{k}'$ onto first $l+1$ components, and $\partial^{|\textbf{k}|} / \partial z^{\textbf{k}}$, $\partial^{|\textbf{k}'|} / \partial \overline{z}^{\textbf{k}'}$ are the multinomial notations for the differentiations.
		When we don't precise $k \in \nat$, by our convention, this means $k=0$.
		\par b) \cite[Definition 2.1]{FreixTh} A function $f : Y \setminus F \to \comp$ has \textit{log-log growth on $Y$, with singularities along $F$} (resp. \textit{has logarithmic singularities along $F$}), if for any $x \in Y$, for some adapted chart $(U; z_0, \ldots, z_q; l)$ of $F$ at $x$, and for some $C > 0$, $p \in \nat$, we have
		\begin{equation}
			\begin{aligned}
				& \textstyle |f(z_0, \ldots, z_q)| \leq C \Big( \ln \big| \ln \big( \max_{k = 0}^{l} \{ |z_k| \} \big) \big| \Big)^p  + C,
				\\
				& \text{\Big( resp. } 
					\textstyle |f(z_0, \ldots, z_q)| \leq C \big| \ln \big( \max_{k = 0}^{l} \{ |z_k| \} \big) \big|^p  + C
				\text{\Big).}
			\end{aligned}
		\end{equation}
		\par c) A function $f : Y \setminus |D_0| \to \comp$ has \textit{logarithmic singularities of order $k \in \nat$ along $D_0$}, if $f \in \ccal^{\infty}(Y / |D_0|, \comp)$, and for any $\textbf{k}, \textbf{k}' \in \nat^{q+1}$, for some adapted chart $(U; z_0, \ldots, z_q; l)$ of $D_0$ at $x \in |D_0|$, and some $C > 0$, $p \in \nat$, we have
		\begin{equation}\label{eq_loglog_grwth_cond}
			\Big|
			\frac{\partial^{|\textbf{k}|}}{\partial z^{\textbf{k}}} 
			\frac{\partial^{|\textbf{k'}|}}{\partial \overline{z}^{\textbf{k}'}}
			f(z_0, \ldots, z_q)
			\Big| 
			\leq 
			C |z|^{-|\textbf{k}_0| - |\textbf{k}'_0|} \prod_{k=0}^{l} | \ln |z_k| |^p  + C,
		\end{equation}
		where $\textbf{k}_0, \textbf{k}'_0 \in \nat^{l+1}$ are the projections of $\textbf{k}, \textbf{k}'$ onto first $l+1$ components.
		\par d) \cite[p. 240]{MumHirz} A differential form over $Y \setminus |D_0|$ has \textit{Poincaré growth on $Y$, with singularities along $D_0$}, if it can be expressed as a linear combination of monomials constructed using $d \zeta_k, \overline{d \zeta_k}$, $k = 0, \ldots, q$ with coefficients $f \in \ccal^{\infty}(Y \setminus |D_0|) \cap L^{\infty}(Y \setminus |D_0|)$.
		\par e) A current $T$ over $Y \setminus |D_0|$ has \textit{log-log growth (resp. Poincaré growth) on $Y$ with singularities along $D_0$} if it is represented by integration of a $L^1_{\rm{loc}}$-form and there is a form $\alpha$ with log-log growth (resp. Poincaré growth) on $Y$, with singularities along $D_0$, such that $-\alpha \leq T \leq \alpha$.
		\par f) \cite[Definition 7.1]{BurKrKunChow} A differential form $\alpha$ is \textit{pre-log-log of order $k \in \nat$ on $Y$, with singularities along $D_0$}, if $\alpha$, $\partial \alpha$, $\dbar \alpha$, $\partial \dbar \alpha$ have log-log growth of order $k$ on $Y$, with singularities along $D_0$. Again, when the order is not precised, by our convention, we suppose $k = 0$. 
		\par g) \cite[Definition 2.14]{FreixTh} A smooth function $f: Y \setminus |D_0| \to \comp$ is \textit{P-singular, with singularities along $D_0$}, if $\partial f$, $\dbar f$, $\partial \dbar f$ have Poincaré growth on $Y$, with singularities along $D_0$. 
	\end{defn}
	\begin{prop}[{Burgos Gil-Kramer-K\"uhn \cite[Proposition 7.6]{BurKrKunChow}}]\label{prop_bur_kr_kuhn}
		a) Any differential form $Y \setminus |D_0|$ with log-log growth with singularities along $D_0$ is locally integrable.
		\par b) If $\alpha$ is pre-log-log form on $Y$ with singularities along $D_0$, then
		\begin{equation}
			[d \alpha]_{L^1} = d[\alpha]_{L_1}.
		\end{equation}
	\end{prop}
	\begin{defn}\label{defn_loglog_gr_lin}
		Let $L$ be a holomorphic line bundle over $Y$ and let $h^L$ be a smooth Hermitian metric on $L$ over $Y \setminus (F \cup |D_0|)$. We say that the metric $h^L$ \textit{has log-log growth with singularities along $F \cup |D_0|$} if for any local holomorphic frame $\upsilon$ of $L$, the function $\ln h(\upsilon, \upsilon)$ , has log-log growth on $Y$, with singularities along $F \cup |D_0|$.
	\end{defn}
	\begin{defn}[{\cite[p. 242]{MumHirz}, {\cite[Definition 4.29]{BurKrKun}} (cf. \cite[Definition 3.1]{FreixTh})}]\label{defn_pre_loglog}
		Let $\xi$ be a holomorphic vector bundle over $Y$ and let $h^{\xi}$ be a Hermitian metric on $\xi$ over $Y \setminus |D_0|$. We say that the metric $h^{\xi}$ \textit{is pre-log-log of order $k \in \nat$ (resp. good) with singularities along $D_0$} if for any local holomorphic frame $e_1, \ldots, e_{\rk{\xi}}$ of $\xi$,
		the functions $h^{\xi}(e_i, e_j)$, $i,j \in 1, \ldots, \rk{\xi}$, $(\det (H))^{-1}$, for $H = (h^{\xi}(e_i, e_j))_{i,j = 1}^{\rk{\xi}}$ have logarithmic singularities of order $k \in \nat$ (resp. order 0), and the entries of the matrix $(\partial H) H^{-1}$, are pre-log-log  of order $k \in \nat$ (resp. P-singular), with singularities along $D_0$.
		Again, when the order is not precised, by our convention, we suppose $k = 0$.
	\end{defn}

	For line bundles, we have the following easy criteria of pre-log-log and good conditions:
	\begin{prop}[{\cite[Proposition 3.2]{FreixTh}}]\label{prop_crit_preloglog}
		Let $L$ be a holomorphic line bundle over $Y$ and let $h^L$ be a smooth Hermitian metric on $L$ over $Y \setminus |D_0|$.
		Then $h^L$ is pre-log-log (resp. good) with singularities along $D_0$ if and only if for every local holomorphic frame $\upsilon$ of $L$ over $U \subset Y$, the function $\ln h^L(\upsilon, \upsilon)$ is pre-log-log (resp. P-singular), with singularities along $D_0$.
	\end{prop}
	
	The following proposition explains how nice vector bundles can be used to precise the extension of a given line bundle.
	\begin{prop}\label{prop_only_ext_nice}
		Let $(L', h^{L'})$ be a continuous Hermitian line bundle over $Y \setminus |D_0|$. 
		Then there is at most one holomorphic line bundle $L$ over $Y$ which extends $L'$ in such way that the Hermitian metric $h^{L'}$ becomes nice with singularities along $D_0$ in the sense of Definition \ref{defn_nice}b).
	\end{prop}
	\begin{proof}
		The proof is similar to the proof of the statement about the \textit{good vector bundles} proved by Mumford in \cite[Proposition 1.3]{MumHirz} with only one change: for an open $U \subset Y$, the frames of $L$ are given by
		\begin{equation}\label{eq_nice_frame_ch}
			\Big\{ 
				s \in \ccal^{\infty}(U \setminus |D_0|, L') : \ln h^{L'}(s,s) \text{ has log-log singularities along } D_0
			\Big\}.
		\end{equation}
		From now on, the proof repeats \cite[Proposition 1.3]{MumHirz}, and we leave it to the interested reader.
	\end{proof}
	\begin{proof}[Proof of  Corollary \ref{cor_ext_character} modulo Theorem \ref{thm_cont}.]
		By Theorems \ref{thm_cont}2, \ref{thm_cont}3 we see that the extension $\mathscr{L}_n$ satisfies the required properties. Now, Proposition \ref{prop_only_ext_nice} shows that there is at most one extension, which finishes the proof.
 	\end{proof}
		
	\subsection{Bott-Chern currents and pre-log-log Hermitian vector bundles}\label{sect_bc_curren}
	In this section we recall the theory of Bott-Chern currents associated with pre-log-log Hermitian vector bundles, which was implicit in the paper \cite{BurKrKun} and which extends previous work of Bismut-Gillet-Soulé \cite[Theorem 1.29]{BGS1}.
	We work with a complex manifold $Y$, and a normal crossing divisor $D_0 \subset Y$.
	 \par Now, let's define the vector space 
	 \begin{equation}	
	 	P_{{\rm{pll}}}(Y, D_0) := \oplus_{i} P_{{\rm{pll}}}^{(i,i)}(Y, D_0),
	 \end{equation}
	 where $P_{{\rm{pll}}}^{(i,j)}(Y, D_0)$ are the vector spaces of pre-log-log differential forms on $Y$ of degree $(i,j)$ with singularities along $D_0$. 
	 We denote 
	\begin{equation}
		\begin{aligned}
			& P'_{{\rm{pll}}}{}^{(i,i)}(Y, D_0) = \partial (P_{{\rm{pll}}}^{(i-1,i)}(Y, D_0)) + \dbar (P_{{\rm{pll}}}^{(i,i-1)}(Y, D_0)), \\
			& P'_{{\rm{pll}}}(Y, D_0) := \oplus_{i} P'_{{\rm{pll}}}{}^{(i,i)}(Y, D_0).
		\end{aligned}
	\end{equation}
	 Then, by Definition \ref{defn_loglog_gr}f), we have $P'_{{\rm{pll}}}(Y, D_0)  \subset P_{{\rm{pll}}}(Y, D_0)$.
	 \par Let $\phi: {\rm{M}}_k(\comp) \to \comp$ be a polynomial map, which is invariant under conjugation. 
	 For a Hermitian vector bundle $(\xi, h^{\xi})$ over $Y$, we denote by $\phi(\xi, h^{\xi})$ the differential form, constructed by $\phi(\xi, h^{\xi}) = \phi(-  \frac{1}{2 \pi \imun} \cdot R^{h^{\xi}})$, where $R^{h^{\xi}}$ is the curvature of the Chern connection on $(\xi, h^{\xi})$. 
	\begin{prop}\label{prop_ch_classes_preloglog}
		Let $h^{\xi}$ be a Hermitian metric on $\xi$, which is pre-log-log of order $k \in \nat$ (resp. good) over $Y \setminus D_0$  with singularities along $D_0$. 
		Then the form $\phi(\xi, h^{\xi})$ is pre-log-log of order $k$ (resp. good) on $Y$ with singularities along $D_0$, and the induced current $[\phi(\xi, h^{\xi})]_{L^1}$ represents the cohomology class of $\phi(\xi) \in H^{2k}(Y, \comp)$.
	\end{prop}
	\begin{proof}
		This was proved for good Hermitian metrics by Mumford in \cite[Theorem 1.4]{MumHirz}. The proof for pre-log-log Hermitian metrics remains identical, see \cite[Proposition 7.25]{BurKrKunChow}.
	\end{proof}
	 \par We consider a short exact sequence of vector bundles over $Y$, endowed with Hermitian metrics over $Y \setminus D_0$, which are pre-log-log with singularities along $D_0$. 
	 \begin{equation}\label{eq_sh_exct_seq}
		\begin{CD}
			0 @>>> (E_0, h^{E_0}) @>>> (E_1, h^{E_1}) @>>> (E_2, h^{E_2}) @>>> 0.
		\end{CD}
	\end{equation}
	\begin{thm}\label{thm_bc_preloglog}
		For $\phi$ as above, there is a unique way to attach to every	 exact sequence $(E_{\bullet}, h^{E_{\bullet}})$ as in (\ref{eq_sh_exct_seq}) a class $\tilde{\phi}(E_{\bullet}, h^{E_{\bullet}}) \in P_{{\rm{pll}}}(Y, D_0) / P'_{{\rm{pll}}}(Y, D_0)$ such that
		\par a) $ \frac{1}{2 \pi \imun} \partial \dbar \tilde{\phi}(E_{\bullet}, h^{E_{\bullet}}) = \phi(E_0 \oplus E_2, h^{E_0} \oplus h^{E_2}) - \phi(E_1, h^{E_1})$.
		\par b) If (\ref{eq_sh_exct_seq}) induces an isometry $(E_1, h^{E_1}) = (E_0 \oplus E_2, h^{E_0} \oplus h^{E_2})$, then $\tilde{\phi}(E_{\bullet}, h^{E_{\bullet}}) = 0$.
		\par c) If $Y'$ is another complex manifold, $D'_0$ is a normal crossing divisor in $Y'$, and $f : Y' \to Y$ is a holomorphic map
 such that $f^{-1}(D_0) \subset D'_0$, then $\tilde{\phi}(f^* E_{\bullet}, f^* h^{E_{\bullet}}) = f^* \tilde{\phi}(E_{\bullet}, h^{E_{\bullet}})$.
 		\par d) For a Hermitian exact square 
 		\begin{equation}\label{eq_sh_exct_seq}
		\begin{CD}
			@.  0 @. 0 @. 0 @.
			\\
			@. @VVV @VVV @VVV @.
			\\
			0 @>>> (E_0^{0}, h^{E_0}) @>>> (E_1^{0}, h^{E_1^{0}}) @>>> (E_2^{0}, h^{E_2^{0}}) @>>> 0
			\\
			@. @VVV @VVV @VVV @.
			\\
			0 @>>> (E_0^{1}, h^{E_0^{1}}) @>>> (E_1^{1}, h^{E_1^{1}}) @>>> (E_2^{1}, h^{E_2^{1}}) @>>> 0
			\\
			@. @VVV @VVV @VVV @.
			\\
			0 @>>> (E_0^{2}, h^{E_0^{2}}) @>>> (E_1^{2}, h^{E_1^{2}}) @>>> (E_2^{2}, h^{E_2^{2}}) @>>> 0
			\\
			@. @VVV @VVV @VVV @.
			\\
			@.  0 @. 0 @. 0 @.
		\end{CD}
	\end{equation}
	we have in $P_{{\rm{pll}}}(Y, D_0) / P'_{{\rm{pll}}}(Y, D_0)$:
	\begin{equation}
		\tilde{\phi}(E^1_{\bullet}, h^{E^1_{\bullet}}) - \tilde{\phi}(E^0_{\bullet} \oplus E^2_{\bullet}, h^{E^0_{\bullet}} \oplus h^{E^2_{\bullet}}) - \tilde{\phi}(E_1^{\bullet}, h^{E_1^{\bullet}})  + \tilde{\phi}(E_0^{\bullet} \oplus E_2^{\bullet}, h^{E_0^{\bullet}} \oplus h^{E_2^{\bullet}}) = 0.
	\end{equation}
	\end{thm}
	\begin{proof}
		This theorem was proved in \cite[Theorem 4.64]{BurKrKun} for the pre-log-log vector bundles of infinite order, but their proof remains valid for pre-log-log vector bundles. This is due to the fact that in \cite[Theorem 4.64]{BurKrKun}, the estimates of the \textit{higher} derivatives of the underlying metric are only used in the estimates of \textit{higher} derivatives of the Bott-Chern form.
	\end{proof}
	\begin{rem}\label{rem_bc_main}
		a) By \cite[Theorem 1.29]{BGS1}, we see that if $h^{E_i}$ are smooth, then $\tilde{\phi}(E_{\bullet}, h^{E_{\bullet}})$ coincides with the Bott-Chern form constructed in \cite{BGS1}.
		\par b) When the exact sequence $(\xi_{\bullet}, h^{\xi_{\bullet}})$ consists of two elements
		\begin{equation}\label{eq_sh_exct_seq2}
		\begin{CD}
			(\xi_{\bullet}, h^{\xi_{\bullet}}) : 0 @>>> (\xi, h^{\xi}) @>>> (\xi, h^{\xi}_0) @>>> 0,
		\end{CD}
	\end{equation}
	we denote for simplicity
	\begin{equation}
		\tilde{\phi}(\xi, h^{\xi}, h^{\xi}_{0}) := \tilde{\phi}(\xi_{\bullet}, h^{\xi_{\bullet}}).
	\end{equation}
	By Theorem \ref{thm_bc_preloglog}a), we have
	\begin{equation}\label{eq_bc_two_elem}
		\frac{\partial \dbar}{2 \pi \imun} \tilde{\phi}(\xi, h^{\xi}, h^{\xi}_{0}) = \phi(\xi, h^{\xi}) - \phi(\xi, h^{\xi}_{0}).
	\end{equation}
	In this paper we will only need to consider the short exact sequences consisting of $2$ terms.
	Nevertheless, we state the theorem for general short exact sequences, since in the forthcoming paper \cite{FinII4} about Deligne-Mumford isometry, we use it in full generality.
	\par c) It's disputable if one needs the last axiom. In the original set of the axioms \cite[Theorem 1.29]{BGS1} for smooth metrics, the last axiom is shown to be a consequence of the first three \cite[Theorem 1.20, Corollary 1.30]{BGS1}.
	\end{rem}
	
	\begin{prop}\label{prop_bc_ch_tilde_0_seq}
		Let $(E_{\bullet}, h^{E_{\bullet}})$ be the exact sequence from (\ref{eq_sh_exct_seq}). Taking into account the isomorphism $\det E_1 \simeq \det E_0 \otimes \det E_2$, we have
		\begin{equation}\label{eq_tilde_ch_0_seq}
			\widetilde{\ch}(E_{\bullet}, h^{E_{\bullet}})^{[0]} = \ln \det \big( (h^{E_0} h^{E_2}) / h^{E_1} \big).
		\end{equation}
	\end{prop}
	\begin{proof}
		First of all, we have
		\begin{equation}\label{eq_ch_tilde_0_aux_seq}
			\widetilde{\ch}(E_{\bullet}, h^{E_{\bullet}})^{[0]} = \widetilde{c_1}(E_{\bullet}, h^{E_{\bullet}}).
		\end{equation}
		By the uniqueness of the Bott-Chern classes from Theorem \ref{thm_bc_preloglog}, it is enough to see that (\ref{eq_tilde_ch_0_seq}) satisfies the requirements of Theorem \ref{thm_bc_preloglog} for $\phi = \rm{Tr}$, i.e. representing the first Chern form.
	\end{proof}
	\begin{prop}\label{prop_bc_ex_calc}
		Let $\xi$ be a holomorphic vector bundle over $Y$, and let $h^{\xi}$, $h^{\xi}_0$ be pre-log-log Hermitian metrics on $\xi$ with singularities along $D_0$. Then
		\begin{equation}\label{eq_ch_tilde_0}
			\widetilde{\ch}(\xi, h^{\xi}, h^{\xi}_{0})^{[0]} = \ln \det(h^{\xi} / h^{\xi}_{0}).
		\end{equation}
		Moreover, if $\xi$ is of rank $1$, then we have
		\begin{equation}\label{eq_ch_tilde_2}
			\widetilde{\ch}(\xi, h^{\xi}, h^{\xi}_{0})^{[2]} = \frac{1}{2} \ln \det(h^{\xi} / h^{\xi}_{0}) \cdot \Big( c_1(\xi, h^{\xi}_{0}) + c_1(\xi, h^{\xi}) \Big).
		\end{equation}
	\end{prop}
	\begin{proof}
		First of all, the identity (\ref{eq_ch_tilde_0}) follows directly from Proposition \ref{prop_bc_ch_tilde_0_seq}.
		\par
		Now, we note that for smooth metrics (\ref{eq_ch_tilde_2}) follows from \cite[Theorems 1.27, 1.30]{BGS1}.
		To prove (\ref{eq_ch_tilde_2}) in full generality, we point out that two pre-log-log Hermitian metrics $h^{\xi}_{0}$, $h^{\xi}_{1}$ on holomorphic \textit{line} bundles can be joined by a family of uniformly pre-log-log Hermitian line bundles $h^{\xi}_{t}$ such that $(h^{\xi}_{t})^{-1} \partial_t h^{\xi}_{t}$ has uniform log-log growth. Take for example $h^{\xi}_{t} := (h^{\xi})^{t} (h^{\xi}_{{\rm{sm}}})^{1-t}$ for some smooth Hermitian metric $h^{\xi}_{{\rm{sm}}}$.
		 Thus, the construction from  \cite[\S e)]{BGS1} works perfectly well for pre-log-log Hermitian \textit{line} bundles.
		Thus, the reasoning of \cite[Theorems 1.27, 1.30]{BGS1} still holds in the pre-log-log case, and this implies (\ref{eq_ch_tilde_2}) for pre-log-log Hermitian metrics.
	\end{proof}

\section{Regularity and singularities: a proof of Theorem \ref{thm_cont}}\label{sect_pf_cont} 
In this section we prove Theorem \ref{thm_cont}. 
More precisely, in Section 3.1, we prove a technical proposition about the regularity of a push-forward of a differential form in f.s.o.
In Section 3.2, we recall the necessary prerequisites for the proof of Theorem \ref{thm_cont}: the \textit{compactification theorem} \cite[Theorem A]{FinII1}, the \textit{anomaly formula for surfaces with cusps} \cite[Theorem B]{FinII1}, and the result of Bismut-Bost \cite[Théorème 2.2]{BisBost}, describing the asymptotics of the Quillen norm associated with a smooth metric over the total space near the singular fibers. In Section 3.3, we use it to prove Theorem \ref{thm_cont}.
\subsection{Pushforward of differential forms in f.s.o.}
	In this section we will study the singularities of a pushforward of a differential form in f.s.o. 
	This study will be used extensively in the proof of Theorem \ref{thm_cont}.
	The main result of this section is
\begin{prop}\label{prop_int_loglog}
	Let $\pi : X \to S$ be a f.s.o., and let $D \subset X$ be a divisor intersecting $\pi^{-1}(\Delta)$ transversally and such that $\pi|_D: D \to S$ is locally an isomorphism. 
	\par a) Let $\alpha$ be a smooth $2$-form over $X \setminus |D|$ with log-log growth of infinite order along $D$. Then the function $\pi_{*}[\alpha]$ over $S \setminus |\Delta|$ is \textit{very nice} with singularities along $\Delta$ (cf. Definition \ref{defn_nice}).
	\par b)
	Suppose that $\Delta$ has normal crossings.
	Let $\alpha$ be a pre-log-log differential $2$-form over $X \setminus (\pi^{-1}(|\Delta|) \cup |D|)$, with singularities along $\pi^{-1}(\Delta) \cup D$.
	Then the function $\pi_{*}[\alpha]$ over $S \setminus |\Delta|$ is \textit{nice} with singularities along $\Delta$ (cf. Definition \ref{defn_nice}).
	\par c)
	Suppose that $\Delta$ has normal crossings.
	Let $\alpha$ be a differential $(1,1)$-form over $X \setminus (\Sigma_{X/S} \cup |D|)$, such that it has Poincaré growth on $X \setminus |D|$ with singularities along $\pi^{-1}(\Delta)$, and the coupling of $\alpha$ with smooth vertical vector fields over $X \setminus (\Sigma_{X/S} \cup |D|)$ is continuous over and has log-log growth with singularities along $|D|$ (cf. Definition \ref{defn_loglog_gr}b)).
	Let $f : X \setminus (\Sigma_{X/S} \cup |D|) \to \real$ be a continuous function, with log-log growth along $\Sigma_{X/S} \cup |D|$.
	Then the function $\pi_{*}[f \alpha]$ extends continuously over $S$.
\end{prop}
\begin{proof}
	\textbf{The Proposition \ref{prop_int_loglog}a)} was proved by Igusa \cite{Igusa} in the case $D = \emptyset$ (for precisely this version, see Bismut-Bost \cite[Théorèmes 12.2, 12.3]{BisBost}).
	Now let's describe the proof for $D \neq \emptyset$.
	\par We take $t_0 \in S$, and let $U \subset S$ be a small neighbourhood of $t_0$ such that $\pi|_{D}:D \to S$ is an isomorphism on each connected component over $\pi^{-1}(U)$. For simplicity, we suppose that $D|_{\pi^{-1}(U)}$ has only one connected component. We choose coordinates $(z_0, \ldots, z_q)$ of $(\pi|_D)^{-1}(t_0) \in \pi^{-1}(U)$ and $(w_1, \ldots, w_q)$ of $U$ as in (\ref{eq_pr_nonsing}), such that $D$ is given by the equation $\{z_0 = 0\}$ in $\pi^{-1}(U)$.
	\par  
	For $c > 0$ small enough, we denote $V_{1,c} = \{ x \in U : |z_0(x)| < c\}$, and decompose the integration over the fiber in $\pi_{*}[\alpha]$ into two parts: $(\pi|_{V_{1,c}})_{*}[\alpha]$ and $(\pi|_{\pi^{-1}(U) \setminus V_{1,c}})_{*}[\alpha]$. The function $(\pi|_{\pi^{-1}(U) \setminus V_{1,c}})_{*}[\alpha]$ induces \textit{very nice} Hermitian metric on $S \setminus |\Delta|$ with singularities along $\Delta$ by the mentioned result of Igusa.
	\par Let's treat the first part. Trivially, for any $p \in \integ$, we have
	\begin{multline}\label{eq_bound_loglog_poinc}
		\int_{|z_1| < c} \frac{ ( \ln | \ln|z_1||)^p \imun dz_1 d \overline{z}_1}{|z_1 \ln|z_1||^2} = 4 \pi \int_{0}^{c} \frac{(\ln |\ln (r)|)^p dr}{r (\ln (r))^2}  \\
		= 4 \pi \int_0^{-2/\ln(c)} \ln(y)^p d y < + \infty.
	\end{multline}
	Since $\alpha$ is pre-log-log over $\pi^{-1}(U) \setminus  |D|$, with singularities along $D$, by Lebesgue dominated convergence theorem and (\ref{eq_bound_loglog_poinc}), we deduce that 
	\begin{equation}\label{eq_proj_cont}
		(\pi|_{V_{1,c}})_{*}[\alpha] \quad \text{is continuous.}
	\end{equation}
	By taking horizontal derivatives with respect to the coordinates entering the definition of log-log growth of \textit{infinite} order of $\alpha$, we deduce in the same way that the form $(\pi|_{V_{1,c}})_{*}[\alpha]$ is smooth, which concludes the proof.
	\par
	\textbf{Now let's prove \ref{prop_int_loglog}b).} It is essentially the repetition of the proof of \cite[Theorem 5.1.3]{FreixTh} from the thesis of Freixas.
	\par
	We take $t_0 \in S \setminus |\Delta|$, and let $U \subset S \setminus |\Delta|$ be a small neighbourhood of $t_0$  as in the previous case. As before, we suppose for simplicity that $D|_{\pi^{-1}(U)}$ has only one connected component. We choose coordinates $(z_0, \ldots, z_q)$ of $(\pi|_D)^{-1}(t_0) \in \pi^{-1}(U)$ and $(w_1, \ldots, w_q)$ of $U$ as before, and let $V_{1,c}$ be defined as above. We decompose the integration over the fiber in $\pi_{*}[\alpha]$ into two parts: $(\pi|_{V_{1,c}})_{*}[\alpha]$ and $(\pi|_{\pi^{-1}(U) \setminus V_{1,c}})_{*}[\alpha]$.
	\par Trivially, as $\pi|_{\pi^{-1}(U) \setminus V_{1,c}}$ is a submersion, and the form $\alpha$ is continuous over $\pi^{-1}(U) \setminus V_{1,c}$ for any $c > 0$, the form $(\pi|_{\pi^{-1}(U) \setminus V_{1,c}})_{*}[\alpha]$ is continuous over $U$ for any $c > 0$. 
	By this and  the identity
	\begin{equation}\label{eq_lim_c_cont}
		\lim_{c \to 0} (\pi|_{V_{1,c}})_{*}[\alpha] = 0,
	\end{equation}
	which follows from (\ref{eq_bound_loglog_poinc}), we conclude that $\pi_{*}[\alpha]$ is continuous over $S \setminus |\Delta|$. 
	Let's prove that $\pi_{*}[\alpha]$ has log-log growth along $\Delta$.
	\par We fix $t_0 \in |\Delta|$. For simplicity, we suppose that the curve $X_{t_0} = \pi^{-1}(t_0)$ has only one double point singularity at $x_0 \in X_{t_0}$. We choose coordinates $(z_0, \ldots, z_q)$ at $x_0 \in X$ and $(w_1, \ldots, w_q)$ at $t_0 \in S$ as in (\ref{eq_pr_sing}).
	For $c > 0$ small enough, we denote $U = \{ t \in S : |w_1| < c \}$ and 
	\begin{equation}
		V_{2,c} = \{ x \in \pi^{-1}(U): |z_0(x)|, |z_1(x)| < c \}.
	\end{equation}
	\par Let's prove that $(\pi|_{V_{2,c}})_{*}[\alpha]$ has log-log growth along $\Delta$. 
	The divisor $\pi^{-1}(\Delta)$ is given over $V_{2, c}$ by equations $\{z_0 = 0\} + \{z_1 = 0\}$. 
	Let $c < 1/2$.
	We note that since $z_0z_1 = w_1$, the estimates 
	\begin{equation}\label{eq_bound_fiber_coord}
	\ln | \ln |z_0||, \ln | \ln |z_1|| \leq \ln | \ln |w_1||,
	\end{equation}
	are valid in $V_{2, c}$. 
	By (\ref{eq_bound_fiber_coord}), there is a function $f : S \setminus |\Delta| \to \real$ with log-log growth along $\Delta$ such that function $(\pi|_{V_{2,c}})_{*}[\alpha]$ is bounded by
	\begin{equation}\label{eq_st_3}
		f \cdot \int_{H_{w_1, c}} \bigg( \frac{\imun dz_0 d \overline{z}_0}{|z_0 \ln|z_0||^2} + \frac{| dz_0 d \overline{z}_1 |}{|z_0 \ln|z_0|||z_1 \ln|z_1||} + \frac{\imun dz_1 d \overline{z}_1}{|z_1 \ln |z_1||^2} \bigg),	
	\end{equation}
	for $H_{w_1, c} = \{(z_0, z_1) : z_0 z_1 = w_1; |z_0|, |z_1| < c \}$.
	Trivially, there is $C > 0$ such that for any $|\omega_1| < c^2$, we have
	\begin{equation}\label{eq_st_4}
	 \int_{H_{w_1, c}} \frac{| dz_0 d \overline{z}_1 |}{|z_0 \ln|z_0|||z_1 \ln|z_1||} \leq \int_{H_{w_1, c}} \frac{\imun dz_0 d \overline{z}_0}{|z_0 \ln|z_0||^2} < C.
	\end{equation}		
	By (\ref{eq_st_3}) and (\ref{eq_st_4}), we conclude that $(\pi|_{V_{2,c}})_{*}[\alpha]$ has log-log growth along $\Delta$.
	\par Now, as before, for simplicity, we suppose that $D|_{\pi^{-1}(U)}$ has only one connected component. We choose coordinates $(z_0, \ldots, z_q)$ of $\pi^{-1}(U)$ and $(w_1, \ldots, w_q)$ of $U$ as in (\ref{eq_pr_nonsing}), and we conserve the notation $V_{1, c}$ from the previous step. Moreover, we suppose that $\pi^{-1}(\Delta)$ is given by the equation $\{z_1 = 0\}$ over $V_{1,c}$.
	From  (\ref{eq_bound_loglog_poinc}), the fact that $\alpha$ has log-log growth along $\pi^{-1}(\Delta) \cup D$, which is given by $\{z_0 z_1 = 0\}$ in $U$, and the fact that $\pi|_{V_{1,c}}$ is a submersion, we prove that $(\pi|_{V_{1,c}})_{*}[\alpha]$ has log-log growth along $\Delta$.
	\par Finally, as $\pi|_{\pi^{-1}(U) \setminus (V_{1,c} \cup V_{2,c})}$ is a submersion, and the form $\alpha$ has log-log growth along $\pi^{-1}(\Delta)$, the form $(\pi|_{\pi^{-1}(U) \setminus (V_{1,c} \cup V_{2,c})})_{*}[\alpha]$ has log-log growth along $\Delta$. Thus, we deduce that $\pi_{*}[\alpha]$  has log-log growth along $\Delta$.
	\par Now, to prove that $\pi_*[\alpha]$ is nice, we have to study the distributional derivatives $\partial [\pi_*[\alpha]]_{L_1}$, $\dbar [\pi_*[\alpha]]_{L_1}$, $\partial \dbar [\pi_*[\alpha]]_{L_1}$. Let's concentrate on the study of $\partial [\pi_*[\alpha]]_{L_1}$, as the others are similar.
	First of all, by Fubini theorem, since $\pi^{-1}(\Delta)$ is Lebesgue negligible, we have
	\begin{equation}\label{eq_leb_negl}
		\big[\pi_*[\alpha]\big]_{L_1} = \pi_*\big[ [\alpha]_{L_1} \big].
	\end{equation}	 
	By Stokes theorem, we have
	\begin{equation}\label{eq_stok_int}
		\partial \pi_*\big[ [\alpha]_{L_1} \big] = \pi_*\big[ \partial[\alpha]_{L_1} \big].
	\end{equation}
	By the fact that $\alpha$ is pre-log-log, and by Proposition \ref{prop_bur_kr_kuhn}, we have
	\begin{equation}\label{eq_burkrkun_put_diff}
		\partial [\alpha]_{L_1} = [\partial \alpha]_{L_1}.
	\end{equation}
	Thus, by (\ref{eq_leb_negl}), (\ref{eq_stok_int}) and (\ref{eq_burkrkun_put_diff}), we see that it is enough to prove that the differential form $\pi_* [\partial \alpha]$ over $S \setminus |\Delta|$ is continuous and has log-log growth along $\Delta$. 
	The continuity over $S \setminus |\Delta|$ is proved as before. 
	Let's prove that it has log-log growth along $\Delta$. As before, we decompose the integration  $\pi_* [\partial \alpha]$ into three parts: $(\pi|_{V_{1,c}})_* [\partial \alpha]$, $(\pi|_{\pi^{-1}(U) \setminus (V_{1,c} \cup V_{2,c})})_* [\partial \alpha]$ and $(\pi|_{V_{2,c}})_* [\partial \alpha]$. Since $\pi|_{V_{1,c}}$ and $\pi|_{\pi^{-1}(U) \setminus (V_{1,c} \cup V_{2,c})}$ are submersions, the first two parts are treated in the same way as before. Let's concentrate on the last part $(\pi|_{V_{2,c}})_* [\partial \alpha]$.
	\par By (\ref{eq_bound_fiber_coord}) and (\ref{eq_st_4}), similarly to (\ref{eq_st_3}), there is a function $f : S \setminus |\Delta| \to \real$ with log-log growth along $\Delta$ such that the form $(\pi|_{V_{2,c}})_{*}[ \partial \alpha]$ is bounded by
	\begin{equation}\label{eq_st_322}
		f \cdot \int_{H_{w_1, c}} \frac{\imun dz_0 d \overline{z}_0}{|z_0 \ln|z_0||^2} \bigg( \frac{dz_1}{|z_1 \ln|z_1| |} +\frac{d\overline{z}_1}{|z_1 \ln|z_1| |} + \beta \bigg),	
	\end{equation}
	where $\beta$ is some bounded differential form in variables $z_2, \ldots, z_q$.
	Now, by the identity $z_0 z_1 = w_0$ and $|z_0|, |z_1| \leq c$, there is a constant $C > 0$ such that we have
	\begin{multline}\label{eq_st_323}
		\bigg| \int_{H_{w_1, c}} \frac{\imun dz_0 d \overline{z}_0}{|z_0 \ln|z_0||^2} \frac{dz_1}{z_1 |\ln|z_1| |} \bigg|
		\\
		=
		\bigg| \frac{dw_0}{w_0} \bigg|  \int_{H_{w_1, c}} \frac{\imun dz_0 d \overline{z}_0}{|z_0 \ln|z_0||^2 \cdot |\ln|w_0/z_0| |}
		\leq 
		C \bigg|  \frac{dw_0}{w_0 |\ln |w_0||} \bigg|.
	\end{multline}
	By (\ref{eq_st_322}) and (\ref{eq_st_323}), we deduce that the form $(\pi|_{V_{2,c}})_{*}[ \partial \alpha]$ has log-log growth along $\Delta$.
	\par 
	\textbf{Now let's prove \ref{prop_int_loglog}c).}
	By the proof of Proposition \ref{prop_int_loglog}b), we see that $\pi_{*}[f \alpha]$ is continuous over $S \setminus |\Delta|$. 
	Now, let $t_0 \in |\Delta|$, $U \in S$ and $V_{i,c}$, $i=1, 2$, be as before.
	Trivially, since $f \alpha$ is continuous over $\pi^{-1}(U) \setminus (V_{1,c} \cup V_{2,c})$, and $\pi|_{\pi^{-1}(U) \setminus (V_{1,c} \cup V_{2,c})}$ is a submersion, we see that $(\pi|_{\pi^{-1}(U) \setminus (V_{1,c} \cup V_{2,c})})_{*}[f \alpha]$ is continuous for any $c > 0$.
	Let's prove that for $i=1,2$, we have
	\begin{align}
		& \lim_{c \to 0} (\pi|_{V_{i,c}})_{*}[f \alpha] = 0, \label{eq_lim_c_cont2}
	\end{align}
	If (\ref{eq_lim_c_cont2}) holds, we would immediately conclude that $\pi_*[f \alpha]$ is continuous.
	\par 
	By the fact that $(\pi|_{V_{i,c}})_{*}[f \alpha]$ depends only on the coupling of $f \alpha$ with two vertical vector fields, and the fact that those couplings have log-log growth on $X \setminus \Sigma_{X/S}$ with singularities along $D$, we deduce (\ref{eq_lim_c_cont2}) for $i = 1$ from (\ref{eq_bound_loglog_poinc}).
	\par Since $\alpha$ has Poincaré growth on $X \setminus |D|$ with singularities along $\pi^{-1}(\Delta)$, and $f$ has log-log growth along $\Sigma_{X/S}$, we deduce that there are $C > 0$, $p \in \nat$ such that
	\begin{multline}\label{eq_lim_c_cont_3}
		(\pi|_{V_{2,c}})_{*}[f \alpha] \leq C \int_{H_{w_1, c}} \bigg( \frac{(\ln |\ln |z_0||)^p \imun dz_0 d \overline{z}_0}{|z_0 \ln|z_0||^2} 
		\\
		+ \frac{(\ln |\ln |z_0||)^p (\ln |\ln |z_1||)^p | dz_0 d \overline{z}_1 |}{|z_0 \ln|z_0|||z_1 \ln|z_1||} + \frac{(\ln |\ln |z_1||)^p \imun dz_1 d \overline{z}_1}{|z_1 \ln |z_1||^2} \bigg),	
	\end{multline}
	where $H_{w_1, c}$ is as in (\ref{eq_st_3}). 
	By (\ref{eq_bound_loglog_poinc}), (\ref{eq_lim_c_cont_3}) and Cauchy inequality, we deduce (\ref{eq_lim_c_cont2}) for $i = 2$, which finally proves that $\pi_*[f \alpha]$ is continuous over $S$.
\end{proof}
	
	The next proposition explains why Proposition \ref{prop_int_loglog} is well-suited to our Assumptions \textbf{S1}, \textbf{S2}, \textbf{S3}. Let $L$ be a holomorphic line bundle over $X$ and let $h^{L}_{i}$, $i=1,2$ be smooth Hermitian metrics on $L$ over $X \setminus (\pi^{-1}(|\Delta|) \cup |D|)$.
	\begin{prop}\label{prop_bc_current_sat_ass}
		\par a) Suppose that $h^{L}_{i}$, $i=1,2$ extend smoothly over $X \setminus |D|$, and they are pre-log-log of infinite order, with singularities along $D$. 
		Then there is a differential form $\alpha$ in the class $[\widetilde{\ch}(L, h^{L}_{1}, h^{L}_{2})]^{[2]} \in P_{\rm{pll}}(X \setminus \pi^{-1}(|\Delta|), D)$, which satisfies the hypothesis of Proposition \ref{prop_int_loglog}a).
		\par b)
		Suppose that $h^{L}_{i}$, $i=1,2$ are pre-log-log, with singularities along $\pi^{-1}(|\Delta|) \cup |D|$. 
		Then there is a differential form $\alpha$ in the class $[\widetilde{\ch}(L, h^{L}_{1}, h^{L}_{2})]^{[2]} \in P_{\rm{pll}}(X \setminus \pi^{-1}(|\Delta|), D)$, which satisfies the hypothesis of Proposition \ref{prop_int_loglog}b).
		\par c) Suppose that $h^{L}_{i}$, $i=1,2$ extend continuously over $X \setminus (\Sigma_{X/S} \cup |D|)$, have log-log growth with singularities along $\Sigma_{X/S} \cup |D|$, are good in the sense of Mumford on $X \setminus |D|$ with singularities along $\pi^{-1}(\Delta)$, and the coupling with two vertical vector fields of $c_1(L, h^{L}_{i})$, $i = 1,2$ are continuous over $X \setminus (\Sigma_{X/S} \cup |D|)$ and has log-log growth on $X \setminus \Sigma_{X/S}$ with singularities along $D$.
		\begin{sloppypar}
		Then there is a function $f$ and a differential form $\alpha$ such that $f \alpha$ is in the class $[\widetilde{\ch}(L, h^{L}_{1}, h^{L}_{2})]^{[2]} \in P_{\rm{pll}}(X \setminus \pi^{-1}(|\Delta|), D)$, and they satisfy the hypothesis of Proposition \ref{prop_int_loglog}c).
		\end{sloppypar}
	\end{prop}
	\begin{proof}
	\begin{sloppypar}
		It follows directly from Propositions \ref{prop_crit_preloglog}, \ref{prop_bc_ex_calc}. For Proposition \ref{prop_bc_current_sat_ass}c), take $f = [\widetilde{\ch}(L, h^{L}_{1}, h^{L}_{2})]^{[0]}$ and $\alpha = (c_1(L, h^L_{1}) + c_1(L, h^L_{2}))/2$.
	\end{sloppypar}
	\end{proof}

\subsection{Some properties of the Quillen metric} 
	Let $(\overline{M}, D_M, g^{TM})$  be a surface with cusps. Let's recall some notions from \cite{FinII1}.
	\begin{defn}[Flattening of a metric, {\cite[Definition 1.2]{FinII1}}]
		We say that a (smooth) metric $g^{TM}_{\rm{f}}$ over $\overline{M}$ is a \textit{flattening} of $g^{TM}$ if there is $\epsilon > 0$ such that $g^{TM}$ is induced by (\ref{reqr_poincare}) over $V_i^{M}(\epsilon)$, and 
	\begin{equation}\label{fl_exterior}
		g^{TM}_{\rm{f}}|_{M \setminus (\cup_i V_i^{M}(\epsilon))} = g^{TM}|_{M \setminus (\cup_i V_i^{M}(\epsilon))}.
	\end{equation}
	\end{defn}
	Similarly, we defined the notion of \textit{flattening} $\norm{\cdot}_{M}^{\rm{f}}$ for Hermitian norm  $\norm{\cdot}_{M}$. 
	For brevity, we state a version of \cite[Theorem A, Remark 1.4.d)]{FinII1}, which doesn't use the language of \textit{compatible flattenings} from \cite{FinII1}.
	\begingroup
	\setcounter{tmp}{\value{thm}}
	\setcounter{thm}{0} 
	\renewcommand\thethm{\Alph{thm}}
	\begin{thm}[Compact perturbation]\label{thm_comp_appr}
		Let $g^{TM}_{\rm{f}}$, $\norm{\cdot}_{M}^{\rm{f}}$ be some flattenings of $g^{TM}$ and $\norm{\cdot}_{M}$ respectively, then the quantity
		\begin{multline}\label{eqn_of_quil_norms_no_flat}
		2 \rk{\xi}^{-1} \ln \Big( 
			\norm{\cdot}_{Q} \big(g^{TM}, h^{\xi} \otimes \, \norm{\cdot}_{M}^{2n}\big) 
			\big / 
			\norm{\cdot}_{Q} \big(g^{TM}_{\rm{f}}, h^{\xi} \otimes (\, \norm{\cdot}_{M}^{\rm{f}})^{2n})
			\Big)
			\\	
			-
			\rk{\xi}^{-1} \int_M c_1(\xi, h^{\xi}) \Big(2n \ln (\, \norm{\cdot}_{M}^{\rm{f}}/ \norm{\cdot}_{M}) +  \ln (g^{TM}_{\rm{f}} / g^{TM}) \Big)
	\end{multline}
	depends only on the number $n \in \integ$, $n \leq 0$ and the functions $( g^{TM}_{\rm{f}} / g^{TM} )|_{V_i^{M}(1)} \circ (z_i^{M})^{-1} : \dd^* \to \real $ and $(\, \norm{\cdot}_{M}^{\rm{f}}/ \norm{\cdot}_{M} )|_{V_i^{M}(1)}  \circ (z_i^{M})^{-1} : \dd^* \to \real$, for $i = 1, \ldots, m$.
	\end{thm}
	\endgroup
	\setcounter{thm}{\thetmp}
	Now let's recall the anomaly formula for surfaces with cusps, which explains how the Quillen norm changes under the conformal change of the metric with cusps. 
	\par Let's recall that by \cite[Theorem 1.27]{BGS1} (cf. Theorem \ref{thm_bc_preloglog}) and (\ref{eq_ch_td_def}), the Bott-Chern forms of a vector bundle $\xi$ with (smooth) Hermitian metrics $h^{\xi}_{1}$,  $h^{\xi}_{2}$ over $\overline{M}$ satisfy (see also (\ref{eq_ch_tilde_0}), (\ref{eq_ch_tilde_2}))
		\begin{equation}
			 \widetilde{\td}(\xi, h^{\xi}_{1},  h^{\xi}_{2})^{[0]} = \widetilde{\ch}(\xi, h^{\xi}_{1},  h^{\xi}_{2})^{[0]}/2,  \label{ch_bc_0}
		\end{equation}
		If, moreover, $\xi := L$ is a line bundle, we have
		\begin{equation}
			\widetilde{\td}(L, h^L_{1},  h^L_{2})^{[2]} = \widetilde{\ch}(L, h^L_{1}, h^L_{2})^{[2]} / 6. \label{ch_bc_2}
		\end{equation}
	\begin{sloppypar}
	\begingroup
	\setcounter{tmp}{\value{thm}}
	\setcounter{thm}{1} 
	\renewcommand\thethm{\Alph{thm}}
	\begin{thm}[Anomaly formula for surfaces with cusps]\label{thm_anomaly_cusp}
		Let $\phi : M \to \real$ be a smooth function such that
		\begin{equation}
			\text{the triple \quad $(\overline{M}, D_M, g^{TM}_{0} := e^{2 \phi} g^{TM})$ \quad is a surface with cusps.}
		\end{equation}
		We denote by $\norm{\cdot}_M, \norm{\cdot}_{M}^{0}$ the norms induced by $g^{TM}, g^{TM}_{0}$ on $\omega_M(D)$, and by $\norm{\cdot}^W$, $\norm{\cdot}^{W, 0}$ the associated Wolpert norms. 
		Let $ h^{\xi}_0$ be a Hermitian metric on $\xi$ over $\overline{M}$.
		Then the right-hand side of the following equation is finite, and we have
		\begin{equation}\label{eq_anomaly_cusp}
			\begin{aligned}
				2 \ln & \Big(  
				\norm{\cdot}_{Q}  \big(g^{TM}_{0},  h^{\xi}_{0} \otimes (\, \norm{\cdot}_{M}^{0})^{2n} \big) 
				\big/				 
				 \norm{\cdot}_{Q} \big(g^{TM}, h^{\xi} \otimes \, \norm{\cdot}_{M}^{2n} \big) 
				 \Big) 
				\\
				&  = 
		 			  \int_{M} 
		 			\Big( 
	 					\widetilde{\td} \big(\omega_M(D)^{-1}, \, \norm{\cdot}^{-2}_{M}, (\, \norm{\cdot}^{0}_{M})^{-2} \big) \ch \big(\xi, h^{\xi} \big)  \ch \big(\omega_M(D)^n, \norm{\cdot}_M^{2n} \big)  \\
						&  \phantom{= \int_{M} 
		 			\Big[ } +	
		 			\td \big(\omega_M(D)^{-1}, (\, \norm{\cdot}^{0}_{M})^{-2} \big) \widetilde{\ch} \big(\xi, h^{\xi}, h^{\xi}_{0} \big)  \ch \big(\omega_M(D)^n, \norm{\cdot}_M^{2n} \big)  \\
						&  \phantom{= \int_{M} 
		 			\Big[ } +			 		 					
			 			\td \big(\omega_M(D)^{-1}, (\, \norm{\cdot}^{0}_{M})^{-2} \big) \ch \big(\xi, h^{\xi}_{0} \big) \widetilde{\ch} \big(\omega_M(D)^n, \norm{\cdot}_{M}^{2n}, (\, \norm{\cdot}_{M}^{0})^{2n} \big) 									 				
			 		\Big)^{[2]} \\
			 		& \phantom{ = }  - \frac{\rk{\xi}}{6} \ln \Big( \norm{\cdot}^W /  \norm{\cdot}^{W, 0} \Big)	+ \frac{1}{2} \sum \ln \Big(\det (h^{\xi} / h^{\xi}_{0})|_{P_i^{M}} \Big).
			 \end{aligned}
		\end{equation}
	\end{thm}
	\endgroup
	\setcounter{thm}{\thetmp}
	\end{sloppypar}
	\par 
Now, let's recall the result of Bismut-Bost \cite[Théorème 2.2]{BisBost} on the asymptotics of the Quillen norm (see also Bismut \cite{BisDegQuil} for its generalization to higher dimension and Ma \cite{MaFormes03} for the family version of \cite{BisDegQuil}). For this, we fix a f.s.o. $\pi : X \to S$ and smooth Hermitian vector bundles $(\omega_{X / S}, \norm{\cdot}_{X/S}^{\omega, {\rm{sm}}})$, $(\xi, h^{\xi}_{{\rm{sm}}})$ over $X$. We denote by $g^{TX_t}_{{\rm{sm}}}$ the metric on $X_t$, $t \in S \setminus |\Delta|$, induced by $(\omega_{X / S}, \norm{\cdot}_{X/S}^{\omega, {\rm{sm}}})$, and by $\norm{\cdot}_Q(g^{TX_t}_{{\rm{sm}}}, h^{\xi}_{{\rm{sm}}})$ the Quillen norm on $\det (R^{\bullet} \pi_* \xi)^{-1}$.
			 \begin{sloppypar}
			 \begin{thm}[Continuity theorem of Bismut-Bost]\label{thm_BBost_continuous}
			 	The norm $\norm{\cdot}_Q(g^{TX_t}_{{\rm{sm}}}, h^{\xi}_{{\rm{sm}}})^{12} \otimes (\, \norm{\cdot}_{\Delta}^{\rm{div}})^{\rk{\xi}}$ on the line bundle $\det (R^{\bullet} \pi_* \xi)^{-12} \otimes \mathscr{O}_S(\Delta)^{\rk{\xi}}$ over $S \setminus |\Delta|$ is very nice on $S$ with singularities along $\Delta$ in the sense of Definition \ref{defn_nice}.
			 \end{thm}
			 \end{sloppypar}
			 
\subsection{Proof of Theorem \ref{thm_cont}}\label{sect_pf_A}
\begin{sloppypar}
We use the notation from Theorem \ref{thm_cont}.
	Since all the statements are local, it suffices to prove them in a neighbourhood $U$ of $t_0 \in S$.
	We prove them all at the same time in three steps: in \textit{Step 1} we see that by Theorem \ref{thm_anomaly_cusp}, we can trivialize the Poincaré-compatible coordinates associated to $g^{TX_t}$. 
	In \textit{Step 2}, by Theorem \ref{thm_comp_appr}, we reduce the problem to the problem without cusps.
	Finally, in \textit{Step 3}, by the anomaly formula of Bismut-Gillet-Soulé (cf. Theorem \ref{thm_anomaly_cusp}), we reduce the problem to the problem with smooth metrics, which is exactly Theorem \ref{thm_BBost_continuous}. For the proof of Theorem \ref{thm_cont}1, this step is unnecessary since the metrics, which are obtained after Step 2 are already smooth.
	In the first two steps the reduction is done by modifying norms $\norm{\cdot}_{X/S}^{\omega}$, $\norm{\cdot}_{X/S}$ only in the neighbourhood of $|D_{X/S}|$.
\end{sloppypar}
	\par 
	\textbf{Step 1.} 
	Let $V_{i, c}$, $i=1, \ldots, m$, $c > 0$ (resp. $U$) be a neighbourhood of $\sigma_i(t_0)$ (resp. $t_0$) such that for some local coordinates $(z_0, \ldots, z_q)$ of $\sigma_i(t_0)$ and $(w_1, \ldots, w_q)$ of $t_0 \in S$, satisfying (\ref{eq_pr_nonsing}), we have $V_{i, c} = \{ x \in \pi^{-1}(U) : |z_0| < c \}$ and $\{z_0(x) = 0\} = \{ \sigma_i(t): t \in U \}$.
	For simplicity, we note $V_i := V_{i, 1}$.
	Let 
	$\nu : \real_+ \to [0,1]$ be a smooth function satisfying
	\begin{equation}\label{defn_nu}
		\nu(u) = 
		\begin{cases}
			\hfill 1, & \text{if} \quad u < 1/2, \\
			\hfill 0, & \text{if} \quad u > 1. 
		\end{cases}
	\end{equation}
	\par 
	We denote by $\norm{\cdot}^{\omega, 0}_{X/S}$ the norm on $\omega_{X/S}$ over $\pi^{-1}(U \setminus |\Delta|) \setminus |D_{X/S}|$ such that $\norm{\cdot}^{\omega, 0}_{X/S}$ coincides with $\norm{\cdot}^{\omega}_{X/S}$ away from $\cup_i V_i$, and over $(\cup_i V_i) \setminus ( \pi^{-1}(|\Delta|) \cup |D_{X/S}| )$, we have
	\begin{equation}\label{defn_om_0}
		\norm{dz_0}^{\omega, 0}_{X/S} = \big| z_0 \ln |z_0| \big|^{\nu(|z_0|)} \cdot \big( \norm{dz_0}^{\omega}_{X/S} \big)^{1-\nu(|z_0|)}.
	\end{equation}
	Let $\norm{\cdot}^{0}_{X/S}$ be the induced norm on $\omega_{X/S}(D)$ as in Construction \ref{const_norm_div}, and let $g^{TX_t}_{0}$, $t \in S$ be the induced metric with cusps on $X_t$. 
	Then by Construction \ref{const_norm_div} and (\ref{defn_om_0}), we see that if $h^{\xi}$, $\norm{\cdot}_{X/S}$ satisfy Assumptions \textbf{S1} or \textbf{S2} or \textbf{S3}, then $h^{\xi}$, $\norm{\cdot}_{X/S}^{0}$ satisfy Assumptions \textbf{S1} or \textbf{S2} or \textbf{S3} correspondingly.
	We denote by $\norm{\cdot}^{W, 0}_{X/S}$ the Wolpert norm (see Definition \ref{defn_wolp}) on $\otimes_i \sigma_i^{*} \omega_{X/S}$ induced by $g^{TX_t}_{0}$.  
	By Theorem \ref{thm_anomaly_cusp}, for $t \in U \setminus |\Delta|$, we have
	\begin{equation}\label{eq_thmc_aux_anomal1}
	\begin{aligned}
		&\frac{1}{6} \ln \Big(  
			\norm{\cdot}_{Q}  \big(g^{TX_t}_{0},  h^{\xi} \otimes (\, \norm{\cdot}_{X/S}^{0})^{2n}\big)^{12} \otimes 
			\big( \norm{\cdot}_{X/S}^{W, 0} \big)^{-\rk{\xi}}
		\Big)
		\\
		& \qquad \qquad \qquad \qquad  \qquad
		- 
		\frac{1}{6} \ln \Big(  
			\norm{\cdot}_{Q}  \big(g^{TX_t},  h^{\xi} \otimes (\, \norm{\cdot}_{X/S})^{2n}\big)^{12} \otimes 
			\big( \norm{\cdot}_{X/S}^{W} \big)^{-\rk{\xi}}
		\Big)
				\\
				&  = 
		 			\pi_*
		 			\Big[ 
	 					\widetilde{\td} \big(\omega_{X/S}(D)^{-1}, \, \norm{\cdot}^{-2}_{X/S}, (\, \norm{\cdot}^{0}_{X/S})^{-2} \big) \ch \big(\xi, h^{\xi} \big)  \ch \big(\omega_{X/S}(D)^n, \norm{\cdot}_{X/S}^{2n} \big) \\
						&  \phantom{= \int_{M} 
		 			\Big[ } +			 		 					
			 			\td \big(\omega_{X/S}(D)^{-1}, (\, \norm{\cdot}^{0}_{X/S})^{-2} \big) \ch \big(\xi, h^{\xi} \big) \widetilde{\ch} \big(\omega_{X/S}(D)^n, \norm{\cdot}_{X/S}^{2n}, (\, \norm{\cdot}_{X/S}^{0})^{2n} \big) 									 				
			 		\Big]^{[2]}.
			 \end{aligned}
		\end{equation}
	By Propositions \ref{prop_int_loglog}, \ref{prop_bc_current_sat_ass}, we see that the right-hand-side of (\ref{eq_thmc_aux_anomal1}) is very nice on $U \setminus |\Delta|$ with singularities along $\Delta$ under Assumption \textbf{S1}, it is nice on $U \setminus |\Delta|$ with singularities along $\Delta$ under Assumption \textbf{S2}, and it is continuous on $S$ under Assumption \textbf{S3}.
	By this and (\ref{eq_thmc_aux_anomal1}), we see that it is enough to prove Theorem \ref{thm_cont} for the metrics $\norm{\cdot}_{X/S}^{0}$, $\norm{\cdot}_{X/S}^{\omega, 0}$, $\norm{\cdot}_{X/S}^{W, 0}$ instead $\norm{\cdot}_{X/S}$, $\norm{\cdot}_{X/S}^{\omega}$, $\norm{\cdot}_{X/S}^{W}$. 
	\par We note, however, that the norm $\norm{\cdot}^{W, 0}_{X/S}$ is trivial over $U \setminus |\Delta|$, thus, it's enough to prove Theorem \ref{thm_cont} for the norm $\norm{\cdot}_{Q}  \big(g^{TX_t}_{0},  h^{\xi} \otimes (\, \norm{\cdot}_{X/S}^{0})^{2n}\big)^{12} \otimes (\, \norm{\cdot}^{\rm{div}}_{\Delta})^{\rk{\xi}}$ on the line bundle $\det (R^{\bullet} \pi_* (\xi  \otimes \omega_{X/S}(D)^n))^{-12} \otimes \mathscr{O}_S(\Delta)^{\rk{\xi}}$ in place of the norm $\norm{\cdot}^{\mathscr{L}_n}$ on the line bundle $\mathscr{L}_n$.
	\par 
	\textbf{Step 2.}
	We denote $V_i' = V_{i, 1/2} \subset V_i$, and by $\norm{\cdot}^{\omega, {\rm{cmp}}}_{X/S}$ the norm on $\omega_{X/S}$ over $\pi^{-1}(U \setminus |\Delta|)$ such that $\norm{\cdot}^{\omega, {\rm{cmp}}}_{X/S}$ coincides with $\norm{\cdot}^{\omega, 0}_{X/S}$ away from $\cup_i V_i'$, and over $V_i'$, we have
	\begin{equation}\label{defn_om_cmp}
		\norm{dz_0}^{\omega, {\rm{cmp}}}_{X/S} = |z_0 \ln |z_0||^{1-\nu(2|z_0|)} ,
	\end{equation}
	where $\nu : \real \to [0,1]$ is as in (\ref{defn_nu}).
	We denote by $g^{TX_t}_{{\rm{cmp}}}$ the induced metric on  $X_t$.
	We denote by $\norm{\cdot}^{{\rm{cmp}}}_{X/S}$ the norm on $\omega_{X/S}(D)$ over $\pi^{-1}(U \setminus |\Delta|)$, such that $\norm{\cdot}^{{\rm{cmp}}}_{X/S}$ coincides with $\norm{\cdot}^{0}_{X/S}$ away from $\cup_i V_i'$, and over $V_i'$ we have
	\begin{equation}\label{defn_om_cmp2}
		\lVert dz_0 \otimes s_{D_{X/S}}/z_0 \rVert^{{\rm{cmp}}}_{X/S} = |\ln |z_0||^{1-\nu(2|z_0|)}.
	\end{equation}
	By (\ref{defn_om_cmp}), (\ref{defn_om_cmp2}), we see that if $h^{\xi}$, $\norm{\cdot}_{X/S}^{0}$ satisfy Assumptions \textbf{S1} or \textbf{S2} or \textbf{S3}, then $h^{\xi} \otimes (\, \norm{\cdot}_{X/S}^{{\rm{cmp}}})^{2n}$, $\norm{\cdot}_{X/S}^{\omega, {\rm{cmp}}}$ satisfy Assumptions \textbf{S1} or \textbf{S2} or \textbf{S3} for $D_{X/S} = \emptyset$ correspondingly.
	Now, by Theorem \ref{thm_comp_appr}, we see that the function
	\begin{multline}\label{eq_thmc_aux_anomal22}
		2 \ln \Big(  
			\norm{\cdot}_{Q} \big(g^{TX_t}_{\rm{cmp}},  h^{\xi} \otimes (\, \norm{\cdot}_{X/S}^{{\rm{cmp}}})^{2n}\big) 
				\big/				 
				 \norm{\cdot}_{Q} \big(g^{TX_t}_{0}, h^{\xi} \otimes (\, \norm{\cdot}_{X/S}^{0})^{2n}\big) 
				 \Big) 
				\\
				 - \pi_*
		 			\Big[
		 			c_1(\xi, h^{\xi})
		 			\Big( 
	 					n \widetilde{c_1}(\omega_{X/S}(D), (\, \norm{\cdot}_{X/S}^{0})^2,  (\, \norm{\cdot}_{X/S}^{{\rm{cmp}}})^2)
	 					\\
	 					-
	 					\widetilde{c_1}(\omega_{X/S}, (\,\norm{\cdot}_{X/S}^{\omega, 0})^2,  (\, \norm{\cdot}_{X/S}^{\omega, {{\rm{cmp}}}})^2)
	 				\Big)		 				
			 		\Big]^{[2]}
	\end{multline}
	is constant over $U \setminus |\Delta|$.
	By Propositions \ref{prop_ch_classes_preloglog}, \ref{prop_int_loglog}, \ref{prop_bc_current_sat_ass}, we see that the term under the integration in (\ref{eq_thmc_aux_anomal22}) is very nice on $U \setminus |\Delta|$ with singularities along $\Delta$ under Assumption \textbf{S1}, it is nice on $U \setminus |\Delta|$ with singularities along $\Delta$ under Assumption \textbf{S2}, and it is continuous on $S$ under Assumption \textbf{S3}.
	By this and (\ref{eq_thmc_aux_anomal22}), it is enough to prove Theorem \ref{thm_cont} for the metrics $h^{\xi}$, $\norm{\cdot}_{X/S}^{{\rm{cmp}}}$, $\norm{\cdot}_{X/S}^{\omega, {\rm{cmp}}}$ instead of $h^{\xi}$, $\norm{\cdot}_{X/S}^{0}$, $\norm{\cdot}_{X/S}^{\omega, 0}$.
	\par 	\textbf{Step 3.}
	Let $h^{\xi}_{{\rm{sm}}}$, $\norm{\cdot}_{X/S}^{{\rm{sm}}}$, $\norm{\cdot}_{X/S}^{\omega, {\rm{sm}}}$ be some smooth metrics on $\xi$, $\omega_{X/S}(D)$ and $\omega_{X/S}$ respectively over $X$. We denote by $g^{TX_t}_{\rm{sm}}$ the Riemannian metric on $X_t$, induced by $\norm{\cdot}_{X/S}^{\omega, {\rm{sm}}}$.
	By the anomaly formula of Bismut-Gillet-Soulé \cite[Theorem 1.27]{BGS3} (cf. Theorem \ref{thm_anomaly_cusp} for $m = 0$), for $t \in U \setminus |\Delta|$:
	\begin{equation}\label{eq_comp_an_comp_sm}
		\begin{aligned}
				2 \ln & \Big(  
				\norm{\cdot}_{Q} \big(g^{TX_t}_{\rm{cmp}},  h^{\xi} \otimes (\, \norm{\cdot}_{X/S}^{{\rm{cmp}}})^{2n}\big) 
				\big/				 
				 \norm{\cdot}_{Q} \big(g^{TX_t}_{\rm{sm}},  h^{\xi}_{{\rm{sm}}} \otimes (\, \norm{\cdot}_{X/S}^{{\rm{sm}}})^{2n}\big) 
				 \Big) 
				\\
				&  = 
		 			\pi_*
		 			\Big[ 
	 					\widetilde{\td} \big(\omega_M^{-1}, (\, \norm{\cdot}^{{\rm{sm}}}_{M})^{-2}, (\, \norm{\cdot}^{{\rm{cmp}}}_{M})^{-2} \big) \ch \big(\xi, h^{\xi}_{{\rm{sm}}} \big)  \ch \big(\omega_M(D)^n, (\, \norm{\cdot}_M^{{\rm{sm}}})^{2n} \big)  \\
						&  \phantom{= \int_{M} 
		 			\Big[ } +	
		 			\td \big(\omega_M^{-1}, (\, \norm{\cdot}^{{\rm{cmp}}}_{M})^{-2} \big) 
		 			\widetilde{\ch} \big(\xi, h^{\xi}_{{\rm{sm}}}, h^{\xi} \big)  
		 			\ch \big(\omega_M(D)^n, (\, \norm{\cdot}_M^{{\rm{sm}}})^{2n} \big)  \\
						&  \phantom{= \int_{M} 
		 			\Big[ } +			 		 					
			 			\td \big(\omega_M^{-1}, (\, \norm{\cdot}^{{\rm{cmp}}}_{M})^{-2} \big) \ch \big(\xi, h^{\xi} \big) \widetilde{\ch} \big(\omega_M(D)^n, (\, \norm{\cdot}_{M}^{{\rm{sm}}})^{2n}, (\, \norm{\cdot}_{M}^{{\rm{cmp}}})^{2n} \big)
			 	\Big].
		\end{aligned}
	\end{equation}
	By Theorem \ref{thm_bc_preloglog} and Propositions \ref{prop_ch_classes_preloglog}, \ref{prop_int_loglog}, \ref{prop_bc_current_sat_ass}, the right-hand side of (\ref{eq_comp_an_comp_sm}) is very nice on $U \setminus |\Delta|$ with singularities along $\Delta$.
	By this and (\ref{eq_comp_an_comp_sm}), we see that it is enough to prove Theorem \ref{thm_cont} for the metrics $h^{\xi}_{{\rm{sm}}}$, $\norm{\cdot}_{X/S}^{{\rm{sm}}}$, $\norm{\cdot}_{X/S}^{\omega, {\rm{sm}}}$ instead of $h^{\xi}$, $\norm{\cdot}_{X/S}^{{\rm{cmp}}}$, $\norm{\cdot}_{X/S}^{\omega, {\rm{cmp}}}$. 
	But for $h^{\xi}_{{\rm{sm}}}$, $\norm{\cdot}_{X/S}^{{\rm{sm}}}$, $\norm{\cdot}_{X/S}^{\omega, {\rm{sm}}}$, Theorem \ref{thm_cont} follows directly from Theorem \ref{thm_BBost_continuous} by Remark \ref{rem_norm_cross_s1}. Thus, we conclude Theorem \ref{thm_cont}.

\section{Potential theory for log-log currents, a proof of Theorem \ref{thm_curv}}\label{sect_pf_curv} 	 
In Section 4.1 we introduce the potential theory for currents with log-log growth and in Section 4.2 we use it to prove Theorem \ref{thm_curv}. Then we deduce Corollaries \ref{cor_cont_pot}, \ref{corrol_smooth} from Theorem \ref{thm_curv}.

\subsection{Potential theory for currents with log-log growth}\label{sect_pot_th}
	In this section we denote $U = \{(z_1, \ldots, z_q) \in \comp^q : |z_i|<1, \text{ for all } i=1, \ldots, q \}$, and let $D_i \subset U$ be defined by the equation $\{z_i = 0\}$.
	We denote $D = \cup_{i=1}^{l} D_i$, for some $l \leq q$. 
	Before stating the main result of this section, we need the following lemma.
	\begin{lem}\label{lem_cur_ext}
		Let $T$ be a closed $(1,1)$-current over $U \setminus D$ with log-log growth along $D$. Then the trivial $L^1$-extension $[T]_{L^1}$ of $T$, is a closed current over $U$. Also, in a small neighbourhood of $D$, the current $[T]_{L^1}$ can be represented as a difference of two positive closed currents with log-log growth along $D$.
	\end{lem}
	\begin{proof}
		For $p \in \nat$, we denote the functions
		\begin{equation}\label{defn_w_forms}
		\begin{aligned}
			& a_p^{0}(z, w) = (\ln |\ln |z|^2|)^p |w|^2, \\
			& a_p(z, w) = - (\ln |\ln |z|^2|)^p (\ln |\ln |w|^2|)^p,
		\end{aligned}
		\end{equation}	
		over $\comp^2 \setminus  \{0\} \times \comp$  and $\comp^2 \setminus ( \comp \times \{0\} \cup  \{0\} \times \comp )$ respectively. 
		Then it's easy to see that the dominating terms of differential forms $\imun \partial \dbar  a_p^{0}(z, w)$, $\imun \partial \dbar  a_p(z, w)$ are given by
		\begin{align}
			& \frac{-\imun p (\ln |\ln |z|^2|)^{p-1} |w|^2 dz d\overline{z}}{|z \ln |z|^2|^2} + \imun (\ln |\ln |z|^2|)^p dw d\overline{w},  \label{der_f_1} \\
			& \imun p (\ln |\ln |z|^2|)^{p-1} (\ln |\ln |w|^2|)^{p-1} \Bigg( \frac{(\ln |\ln |z|^2|) dw d\overline{w}}{|w \ln |w|^2|^2} +	\frac{ (\ln |\ln |w|^2|) dz d\overline{z}}{|z \ln |z|^2|^2} \Bigg), \nonumber
		\end{align}		 
		respectively. From this, we see that for some open neighbourhood $V \subset U$ of $S$, the forms $\imun \partial \dbar  a_p(z, w)$, $\imun \partial \dbar  (a_p(z, w) + a_p^{0}(z, v))$ are positive on $\comp^2 \setminus ( \comp \times \{0\} \cup  \{0\} \times \comp )$ and $\comp^3 \setminus ( \comp^2 \times \{0\} \cup  \comp \times \{0\} \times \comp \cup \{0\} \times \comp^2 )$ respectively.
		Now, any differential form $\alpha$ over $U \setminus D$ with log-log growth along $D$ can be bounded from above and below by a linear combination of
		\begin{equation}
		\begin{aligned}
			& \frac{\imun (\ln |\ln |z_i|^2|)^{p-1}  dz_j d\overline{z}_j}{|z_j \ln |z_j|^2|^2}, \qquad i, j = 1, \ldots, l, \\
			& \imun (\ln |\ln |z_i|^2|)^p dz_j d\overline{z}_j, \qquad i = 1, \ldots, l; j = 1, \ldots, q. \\
		\end{aligned}
		\end{equation}
		So, since $T$ has log-log growth along $D$, by (\ref{der_f_1}), there are $C>0, p \in \nat$ such that for 
		\begin{equation}\label{eq_defn_W_curr}
			A_p := (q+1)\sum_{i, j=1}^{l} a_p(z_i, z_j) + \sum_{i=1}^{l} \sum_{j=1+l}^{q} a_p^{0}(z_i, z_j),
		\end{equation}
		we have the following inequalities over $V$:
		\begin{equation}\label{eq_bounds_T}
			- C \imun \partial \dbar A_p \leq T \leq C \imun \partial \dbar A_p.
		\end{equation}
		Thus, the current $T + C \imun \partial \dbar A_p$ is closed, positive in $V$, and by (\ref{eq_bound_loglog_poinc}), (\ref{eq_defn_W_curr}), (\ref{eq_bounds_T}) it has finite mass. Thus, by Skoda-El Mir's theorem (cf. \cite[Theorem III.2.3]{DemCompl}), $[T + C \imun \partial \dbar A_p]_{L^1}$ is a closed positive current over $V$. Similarly, $[\imun \partial \dbar A_p]_{L^1}$ is a closed positive current over $V$. So
		\begin{equation}\label{eq_dec_positive}
			[T]_{L^1} = [T + C \imun \partial \dbar A_p]_{L^1} - C [\imun \partial \dbar A_p]_{L^1}
		\end{equation}
		is a closed current over $U$. Also (\ref{eq_dec_positive}) gives the needed decomposition of $[T]_{L^1}$ as a difference of positive currents.
	\end{proof}
	\begin{rem}\label{rem_extension}
		Since our current is locally represented as a difference of two positive currents, its Lelong numbers (cf. \cite[Definition III.5.4]{DemCompl}) are well-defined.
	\end{rem}
The main goal of this section is to prove the following
\begin{prop}\label{prop_pot_th}
	Let $\phi : U \setminus D \to \real$ be a continuous function with log-log growth along $D$. 
	Suppose that for the induced current $[\phi]$ over $U \setminus D$, the current
	\begin{equation}\label{prop_eq_an_1}
		T := \frac{\partial \dbar [\phi]}{2 \pi \imun},
	\end{equation}
	over $U \setminus D$ has log-log growth along $D$. 
	Then we have the following identity of currents over $U$
	\begin{equation}\label{prop_eq_an_3}
		\frac{\partial \dbar [\phi]_{L^1}}{2 \pi \imun} = [T]_{L^1}.
	\end{equation}
\end{prop}
\begin{rem}
	When $l = 1$, this result implies Yoshikawa \cite[Proposition 3.11]{Yoshi2004}, where he obtained this for $T$ of Poincaré growth. If $T$ extends smoothly over $U$, Proposition \ref{prop_pot_th} is a special case of Bismut-Bost \cite[Proposition 10.2]{BisBost}. We note, however, that in our applications, the condition of being pre-log-log and not smooth is essential, see Section \ref{sect_applications}.
\end{rem}
\par
To prove Proposition \ref{prop_pot_th}, we need the following weak analogue of Poincaré lemma for currents of log-log growth:
\begin{lem}\label{lem_pot}
	Let $T$ be a closed $(1,1)$-current over $U$ with log-log growth along $D$. 
	For any $x \in U$, there is a neighborhood $V$ of $x$ and a function $\psi \in L^1_{loc}(V)$ with log-log growth along $D$, satisfying
	\begin{equation}\label{lem46_main_eq}
		\frac{\partial \dbar [\psi]_{L^1}}{2 \pi \imun} = [T]_{L^1}.
	\end{equation}
\end{lem}
\begin{rem}
	This lemma, implies, that Lelong numbers of $T$ (see Remark \ref{rem_extension}) vanish.
\end{rem}
\begin{proof}
	We recall that the functions $A_p : U \setminus D \to \real$, $p \in \nat$ were defined in (\ref{eq_defn_W_curr}).
	Let $C > 0$, $p \in \nat$ be as in (\ref{eq_bounds_T}). By Siu \cite[Proof of Lemma 5.3]{Siu1973}, since the current $[T + C \imun \partial \dbar A_p]_{L^1}$ is closed and positive, there is an open subset $V' \subset U$ and a plurisubharmonic (cf. \cite[Definition B.2.16]{MaHol}) function $R$ over $V'$, such that
	\begin{equation}\label{siu_r_defn}
		\imun \partial \dbar R = \big[T + C \imun \partial \dbar A_p \big]_{L^1}.
	\end{equation}
	Moreover, since
	\begin{equation}
		0 < [2 C \imun \partial \dbar A_p - T]_{L^1} = \imun \partial \dbar (3 C [A_p]_{L^1} - R).
	\end{equation}
	Thus, by plurisubharmonicity (cf. \cite[Proposition A.15]{DinhSib}), there is $C_0 > 0$, such that almost everywhere, we have
	\begin{equation}\label{eq_bound_r_ap}
		-C_0 + 3 C A_p \leq R \leq C_0,
	\end{equation}
	in particular, since $A_p$ has log-log growth along $D$, we deduce by (\ref{eq_bound_r_ap}) that $R - C [A_p]_{L^1}$ has log-log growth along $D$. By (\ref{siu_r_defn}), we get (\ref{lem46_main_eq}) for $\psi := \frac{R - C A_p}{2 \pi \imun}$.
\end{proof}
\begin{proof}[Proof of Proposition \ref{prop_pot_th}]
	Let $\psi$ be a function on $V \subset \subset U$ as in Lemma \ref{lem_pot}, such that
	\begin{equation}\label{eq_psi_defn}
		\frac{\partial \dbar [\psi]_{L^1}}{2 \pi \imun} = [T]_{L^1}.
	\end{equation}
	We denote
	\begin{equation}\label{eq_chi_defn}
		\chi = \phi - \psi.
	\end{equation}
	Then $\chi$ is pluriharmonic on $V \setminus D$ and has log-log singularities along $D$. We'll prove that $\chi$ is pluriharmonic on $V$. Once it will be done, Proposition  \ref{prop_pot_th} will follow from (\ref{prop_eq_an_1}), (\ref{eq_psi_defn}) and (\ref{eq_chi_defn}).
	Let $z' = (0, z'_2, \ldots, z'_q) \in V$ be such that $z' \notin D_i$ for any $i \geq 2$, i.e. $z_i \neq 0$. Then the function
	\begin{equation}
		\chi_0(z) := \chi(z, z'_2, \ldots, z'_q)
	\end{equation}
	is harmonic over $D^*(\epsilon)$, for some $\epsilon > 0$, and has log-log growth along $0 \in D(\epsilon)$.
	By \cite[p. 71-72]{BisBost}, the function $\chi_0$ extends to a harmonic function over $D(\epsilon)$.
	By repeating this for $z'$ in a small neighbourhood of fixed $z'$ in $D_1$, we see that $\chi$ extends over $V \setminus (\cup_{i=2}^{l} D_i)$, such that it's restriction on discs $\{(z, z'_2, \ldots, z'_n) : |z| \leq \epsilon \}$ are harmonic. By the maximum principle and the fact that $\chi$ is smooth over $V \setminus (\cup_{i=1}^{l} D_i)$, we see that this extension is actually locally bounded in $V \setminus (\cup_{i=2}^{l} D_i)$, so by \cite[Theorem 5.24]{DemCompl}, the function $\chi$ is pluriharmonic over $V \setminus (\cup_{i=2}^{l} D_i)$. By repeating this argument for $i = 2, \ldots, l$, we see that $\chi$ is actually pluriharmonic over $V$.
\end{proof}
	
\subsection{Proof of Theorem \ref{thm_curv} and Corollaries \ref{cor_cont_pot}, \ref{corrol_smooth}}\label{sect_thm_d_pf}
	We use the notation from Theorem \ref{thm_curv}. The main ingredients of the proof are Theorems \ref{thm_comp_appr}, \ref{thm_anomaly_cusp}, Proposition \ref{prop_pot_th} and the curvature theorem of Bismut-Bost \cite[Théorème 2.2]{BisBost}, which we now recall.
	\begin{sloppypar}
		 We borrow the notation from Theorem \ref{thm_BBost_continuous}. By Theorem \ref{thm_BBost_continuous}, the Hermitian norm $\norm{\cdot}_Q(g^{TX_t}_{{\rm{sm}}}, h^{\xi}_{{\rm{sm}}})^{12} \otimes (\, \norm{\cdot}_{\Delta}^{\rm{div}} )^{\rk{\xi}}$ is very nice over $S$ with singularities along $\Delta$. In particular, by Remark \ref{rem_nice_chern}, its first Chern form is well-defined.
	\end{sloppypar}
	\begin{thm}[{\cite[Théorème 2.2]{BisBost}}]\label{thm_bgs_curv}
		The following identity of currents over $S$ holds
		\begin{multline}\label{eq_thm_curv_BGS}
			 c_1 \Big( \lambda(j^* \xi)^{12} \otimes \mathscr{O}_{S}(\Delta)^{\rk{\xi}},  \norm{\cdot}_Q(g^{TX_t}_{{\rm{sm}}}, h^{\xi}_{{\rm{sm}}})^{24} \otimes \big( \norm{\cdot}_{\Delta}^{\rm{div}} \big)^{2\rk{\xi}} \Big) 
			\\
			=
			 - 12 \pi_*\Big[\td \big( \omega_{X/S}^{-1}, ( \, \norm{\cdot}_{X/S}^{\omega, {\rm{sm}}} )^{-2} \big) \ch (\xi, h^{\xi}_{{\rm{sm}}})\Big]^{[4]}.
		\end{multline}
	\end{thm}
	\begin{proof}[Proof of Theorem \ref{thm_curv}.]
		 \textbf{Let's treat Assumption \textbf{S1} first.}
	As in Theorem \ref{thm_cont}, the statement is local over the base. 
	So, for any $t_0 \in S$, it is enough to prove (\ref{eq_thm_curv}) in some neighbourhood $U \subset S$ of $t_0$. The fact that the current (\ref{eq_rhs_rrh}) is $L^1_{loc}(S)$ follows from Lebesgue dominated convergence theorem and \cite[Proposition 5.2]{BisBost}. 
	The fact that its closure is a $d$-closed current follows from the fact that it is obtained as a pushforward of a closed form and (\ref{eq_leb_negl}).
	\par By Proposition \ref{prop_bur_kr_kuhn}, (\ref{eq_bc_two_elem}), (\ref{defn_om_0}) and (\ref{eq_thmc_aux_anomal1}), we deduce that 
	\begin{multline}\label{eq_curv_red1}
		c_1 
		\Big( 
			\mathscr{L}_n
			,
			 \norm{\cdot}_Q (g^{TX_t}_{0}, h^{\xi} \otimes (\, \norm{\cdot}_{X/S}^{0})^{2n})^{24} 
			\otimes 
			(\, \norm{\cdot}^{W, 0}_{X/S})^{-2\rk{\xi}}
			\otimes 
			(\, \norm{\cdot}^{\rm{div}}_{\Delta} )^{\rk{\xi}}
		\Big)
		\\
		-
		c_1 
		\Big( 
			\mathscr{L}_n
			,
			\big(  \norm{\cdot}^{\mathscr{L}_n} \big)^{2} 
		\Big)
		=
		- 12 \pi_*\Big[\ch (\xi, h^{\xi})  \Big( \td \big(\omega_{X/S}(D)^{-1}, (\, \norm{\cdot}^{0}_{X/S})^{-2} \big) 
		\\
		\cdot \ch \big(\omega_{X/S}(D)^{n}, (\, \norm{\cdot}^{0}_{X/S})^{2n} \big)
		- \td \big(\omega_{X/S}(D)^{-1}, \norm{\cdot}^{-2}_{X/S} \big) \ch \big(\omega_{X/S}(D)^{n}, \, \norm{\cdot}^{2n}_{X/S} \big) \Big) \Big]^{[4]}.
	\end{multline}
	From (\ref{eq_curv_red1}), wee see that it is enough to prove (\ref{eq_thm_curv}) for the norms $\norm{\cdot}_{X/S}^{\omega, 0}$, $\norm{\cdot}_{X/S}^{0}$ instead of the norms $\norm{\cdot}_{X/S}^{\omega}$, $\norm{\cdot}_{X/S}$.
	\par By (\ref{defn_om_cmp}), (\ref{defn_om_cmp2}), the fact that (\ref{eq_thmc_aux_anomal22}) is constant and Proposition \ref{prop_bur_kr_kuhn}, we deduce
	\begin{equation}\label{eq_curv_red2}
	\begin{aligned}
		&
		c_1 
		\Big( 
			\lambda \big(j^*(\xi  \otimes \omega_{X/S}(D)^n) \big)
			,
			\norm{\cdot}_Q (g^{TX_t}_{\rm{cmp}}, h^{\xi} \otimes (\, \norm{\cdot}_{X/S}^{\rm{cmp}})^{2n})^2
		\Big)
		\\		
		& \qquad \qquad \qquad \qquad \qquad
		-
		c_1 
		\Big( 
			\lambda \big(j^*(\xi  \otimes \omega_{X/S}(D)^n) \big)
			,
			\norm{\cdot}_Q (g^{TX_t}_{0}, h^{\xi} \otimes (\, \norm{\cdot}_{X/S}^{0})^{2n})^2
		\Big)
		\\
		& \qquad
		=
		  \pi_*\Big[\td \big(\omega_{X/S}^{-1}, (\, \norm{\cdot}^{\omega, 0}_{X/S})^{-2} \big) c_1 (\xi, h^{\xi})  \ch \big(\omega_{X/S}(D)^{n}, (\, \norm{\cdot}^{0}_{X/S})^{2n} \big)
		\\
		& 
		\qquad \qquad \qquad \qquad
		- 
		\td \big(\omega_{X/S}^{-1}, (\, \norm{\cdot}^{\omega, {\rm{cmp}}}_{X/S})^{-2} \big) c_1 (\xi, h^{\xi})  \ch \big(\omega_{X/S}(D)^{n}, (\, \norm{\cdot}^{\rm{cmp}}_{X/S})^{2n} \big) \Big]^{[4]}.
	\end{aligned}
	\end{equation}
	Now, since the norms $\norm{\cdot}^{\omega, 0}_{X/S}$, $\norm{\cdot}^{\omega, {\rm{cmp}}}_{X/S}$ and $\norm{\cdot}^{0}_{X/S}$, $\norm{\cdot}^{\rm{cmp}}_{X/S}$ coincide away from $\cup V_i'$, and over $\cup V_i'$ they vary only in the horizontal direction, we deduce that 
	\begin{multline}\label{eq_curv_red3}
			\Big[ 
			\td \big(\omega_{X/S}^{-1}, (\, \norm{\cdot}^{\omega, 0}_{X/S})^{-2} \big)  \ch \big(\omega_{X/S}(D)^{n}, (\, \norm{\cdot}^{0}_{X/S})^{2n} \big)
			\\
			- 
			\td \big(\omega_{X/S}^{-1}, (\, \norm{\cdot}^{\omega, {\rm{cmp}}}_{X/S})^{-2} \big) \ch \big(\omega_{X/S}(D)^{n}, (\, \norm{\cdot}^{\rm{cmp}}_{X/S})^{2n} \big)  \Big]^{[4]} = 0.
	\end{multline}
	Thus, by (\ref{eq_curv_red3}), we can interpret $c_1(\xi, h^{\xi})$ in the right-hand side of (\ref{eq_curv_red2}) as the Chern form $\ch(\xi, h^{\xi})$.
	By Poincaré-Lelong formula and (\ref{eq_ch_td_def}), we deduce that
	\begin{equation}\label{eq_poinc_lel}
		\td \big(\omega_{X/S}^{-1}, (\, \norm{\cdot}^{\omega, 0}_{X/S})^{-2} \big)^{[2]} = \td \big(\omega_{X/S}(D)^{-1}, (\, \norm{\cdot}^{0}_{X/S})^{-2} \big)^{[2]} + \delta_{D_{X/S}}/2. 
	\end{equation}
	Now, by Theorem \ref{thm_bgs_curv} applied for $\xi := \xi \otimes \omega_{X/S}(D)^n$, $h^{\xi} := h^{\xi} \otimes (\, \norm{\cdot}^{\rm{cmp}}_{X/S})^{2n}$ and the metric $g^{TX_t}_{{\rm{cmp}}}$, $t \in U$ induced by $\norm{\cdot}^{\omega, {\rm{cmp}}}_{X/S}$, (\ref{eq_curv_red2}), (\ref{eq_curv_red3}) and (\ref{eq_poinc_lel}), we deduce Theorem \ref{thm_curv} in Assumption \textbf{S1}.
	\par \textbf{Let's treat Assumption \textbf{S2}.}
	First of all, from Proposition \ref{prop_int_loglog}, the current (\ref{eq_rhs_rrh}) has log-log growth along $\Delta$.
	Also, since (\ref{eq_rhs_rrh}) is a pushforward of a closed current, it is a $d$-closed current over $S \setminus |\Delta|$. 
	By Lemma \ref{lem_cur_ext}, the $L^1$-trivial extension of this current is also $d$-closed. 
	Thus, by Theorem \ref{thm_cont}2 and Proposition \ref{prop_pot_th}, we see that it is enough to prove Theorem \ref{thm_curv} over $S \setminus |\Delta|$ without the boundary term $\delta_{\Delta}$.
	\par 	 Now, as before, the statement is local over the base. So, for any $t_0 \in S \setminus |\Delta|$, it is enough to prove (\ref{eq_thm_curv}) in some neighbourhood $U \subset S \setminus |\Delta|$ of $t_0$.
	\par 
	Trivially, (\ref{eq_curv_red1}) still holds over $U$ under Assumption \textbf{S2} over $U$. Similarly (\ref{eq_curv_red2}) also continues to hold over $U$. 
	Thus, by (\ref{eq_curv_red1})-(\ref{eq_poinc_lel}), we deduce that it is enough to prove Theorem \ref{thm_curv} for $\norm{\cdot}_{X/S}^{\omega, \rm{cmp}}$, $\norm{\cdot}_{X/S}^{\rm{cmp}}$ in place of $\norm{\cdot}_{X/S}^{\omega}$, $\norm{\cdot}_{X/S}$.
	However, by Theorem \ref{thm_bgs_curv} (in the current situation it reduces to the special case of the curvature theorem of Bismut-Gillet-Soulé \cite[Theorem 1.9]{BGS3}), we get
	\begin{multline}\label{eq_bgs_final_iden}
		c_1 
		\Big( 
			\lambda \big(j^*(\xi  \otimes \omega_{X/S}(D)^n) \big)
			,
			\norm{\cdot}_Q (g^{TX_t}_{\rm{cmp}}, h^{\xi} \otimes (\, \norm{\cdot}_{X/S}^{\rm{cmp}})^{2n})^2
		\Big)
		\\	
		= - 
		\pi_* \Big[
		\td \big(\omega_{X/S}^{-1}, (\, \norm{\cdot}^{\omega, {\rm{cmp}}}_{X/S})^{-2} \big) \ch (\xi, h^{\xi})  \ch \big(\omega_{X/S}(D)^{n}, (\, \norm{\cdot}^{\rm{cmp}}_{X/S})^{2n} \big) \Big]^{[4]},
	\end{multline}
	which finishes the proof of Theorem \ref{thm_curv} under Assumption \textbf{S2}.
	\end{proof}
		\begin{proof}[Proof of Corollary \ref{cor_cont_pot}]
		Fix local holomorphic frames $\upsilon$, $\upsilon_1, \ldots, \upsilon_m$ of $\mathscr{L}_n$, $\sigma_1^{*} \xi, \ldots, \sigma_m^{*} \xi$ respectively.
		We denote
		\begin{equation}\label{defn_G}
			G: = \big(\, \norm{\upsilon}^{\mathscr{L}_n} \big)^2 \cdot  
			 h^{\det \xi}(\upsilon_1, \upsilon_1)^{6} \cdot \ldots \cdot h^{\det \xi}(\upsilon_m, \upsilon_m)^{6},
		\end{equation}
		where $h^{\det \xi}$ is the metric on the line bundle $\det \xi$, induced by $h^{\xi}$.
		Then by Theorems \ref{thm_cont}2, \ref{thm_curv} and (\ref{eq_c_1_log}), we have the identity
		\begin{equation}\label{eq_cont_pot_curr}
			\frac{\partial \dbar \ln (G)}{2 \pi \imun}
			 =
			 -12 \pi_*\Big[\td(\omega_{X/S}(D)^{-1}, \norm{\cdot}^{-2}_{X/S}) \ch (\xi, h^{\xi})  \ch(\omega_{X/S}(D)^{n}, \norm{\cdot}_{X/S}^{2n} ) \Big].
		\end{equation}
		However, by Theorem \ref{thm_cont}3, the function $G$ is continuous, which finishes the proof by (\ref{eq_cont_pot_curr}).
	\end{proof}
	\begin{proof}[Proof of Corollary \ref{corrol_smooth}]
		It follows from Theorem \ref{thm_cont}3, (\ref{defn_G}), (\ref{eq_cont_pot_curr}) and the regularity theory of elliptic partial differential equations (cf. \cite[Corollary 8.11]{GilTrudBook}).
	\end{proof}

\section{Applications to the moduli space of stable pointed curves}\label{sect_applications}
	In this section we apply the results of Sections \ref{sect_pf_cont}, \ref{sect_pf_curv} to study the Hodge line bundle on  the moduli space of pointed curves. This section is organized as follows: in Section 5.1, we recall the local description of the moduli space $\modulcomp_{g, m}$ of $m$-pointed stable curves of genus $g$ and of the universal projection map $\Omega: \univcurvcomp_{g, m} \to \modulcomp_{g, m}$. Then we recall the definition of the Weil-Petersson metric with Wolpert theorem, expressing it as a push-out of Chern forms under the universal projection map.
	In Section 5.2 we recall the pinching expansion of the hyperbolic metric. From this, we see that the twisted canonical line bundle $\omega_{g, m}(D)$ over $\univcurvcomp_{g, m}$ (see (\ref{def_rel_can_modul})) satisfies Assumptions \textbf{S2}, \textbf{S3}. Then we prove Corollaries \ref{cor_cont_hodge}, \ref{cor_form_TZ_type}, \ref{cor_wp_cont_pot}, \ref{cor_wp_vol}, \ref{cor_del_iso}.
	
	\subsection{Orbifold structure of $\modulcomp_{g, m}$ and $\univcurvcomp_{g, m}$}\label{sect_orb_str}
	We follow closely the expositions of Wolpert \cite{Wolp90}, and we use the notation from Section \ref{sect_intro}.
	\par We fix $\textbf{M} := (\overline{M}, D_M) \in \mathscr{M}_{g, m}$. 
	Let $\Gamma$ be a Fuchsian group of type $(g, m)$ such that $M := \overline{M} \setminus D_M$ is isomorphic to the quotient $\Gamma \setminus \hh$ of the hyperbolic space.
	Recall that the space of Beltrami differentials $H^1(\overline{M}, T^{1,0} \overline{M} \otimes \mathscr{O}_{\overline{M}}(- D_M))$ with the obvious action by the automorphisms group $\rm{Aut}(\textbf{M})$ gives a local chart for $\modul_{g, m}$ in the following way. 
	We take $[\nu_0] \in H^1(\overline{M}, T^{1,0} \overline{M} \otimes \mathscr{O}_{\overline{M}}(- D_M))$. 
	By locally resolving $\dbar$-equation around the cusps, we may choose a representative $\nu \in  \ccal^{\infty}(\overline{M}, \overline{\omega}_{\overline{M}} \otimes T^{1,0} \overline{M} \otimes \mathscr{O}_{\overline{M}}(- D_M))$ in the class $[\nu_0]$, which has compact support in $M := \overline{M} \setminus D_M$. 
	Denote by $\nu_{\hh}$ the pull-back of $\nu$ on $\hh$.
	By a theorem of Ahlfors \cite[Theorem V.5]{Ahlfors_qc}, if $|\nu|_{\ccal^{0}} < 1$, then the Beltrami equation
	\begin{equation}\label{eq_ahl_thm}
		\begin{cases}
			\dbar f^{\nu}(z) = \nu_{\hh}(z) \partial  f^{\nu}(z), & \text{ for }z \in \hh, 
			\\
			\dbar f^{\nu}(z) = \overline{\nu_{\hh}}(\overline{z}) \partial  f^{\nu}(z), & \text{ for } z \in \comp \setminus \hh, 
		\end{cases}
	\end{equation}
	has a unique solution in the class of diffeomorphisms of $\comp \cup \{ \infty \}$, fixing $0, 1, \infty \in \comp \cup \{ \infty \}$.
	We denote
	\begin{equation}
		\Gamma^{\nu} := f^{\nu} \Gamma (f^{\nu})^{-1},
 	\end{equation}
 	then, classically (cf. \cite[p. 69]{Ahlfors_qc}), $\Gamma^{\nu}$ is the Fuchsian group of type $(g, m)$, and $f^{\nu}$ defines a diffeomorphism 
 	\begin{equation}
 		\tilde{f}^{\nu} : \Gamma \setminus \hh \to \Gamma^{\nu} \setminus \hh,
 	\end{equation}
 	which is holomorphic if and only if $[\nu_0] = 0$.
 	\par Now, we choose $\nu_1, \ldots, \nu_N \in \ccal^{\infty}_{c}(M, \overline{\omega}_{\overline{M}} \otimes T^{1,0} \overline{M} \otimes \mathscr{O}_{\overline{M}}(- D_M))$ such that the associated cohomology classes form a basis in $H^1(\overline{M}, T^{1,0} \overline{M} \otimes \mathscr{O}_{\overline{M}}(- D_M))$.
 	Let $c > 0$ be small enough. For $s = (s_1, \ldots, s_N) \in D(c)^N$, we denote
 	\begin{equation}\label{eqn_nu_s_defn}
 		\nu(s) := \sum s_i \nu_i.
 	\end{equation}
 	\par Now, let $(U_{\alpha}, z_{\alpha})$ be an atlas of $\overline{M}$. Then $W_{\alpha} := (\tilde{f}^{\nu(s)} \circ z_{\alpha}, s)$ is a chart mapping $U_{\alpha} \times D(c)^N \to \comp^{N+1}$.
 	This defines a holomorphic atlas on $X_0 := \cup_{s \in S} (\Gamma^{\nu(s)} \setminus \hh)$, for which the obvious projection $\pi_0 : X_0 \to D(c)^N$ is a holomorphic submersion of codimension 1 (cf. \cite[\S 2.4.C]{Wolp90}).
 	Now, since the sections $\nu_1, \ldots, \nu_N$ have compact support in $M$, by (\ref{eq_ahl_thm}), the local coordinate $z_i^{M}$ of $\overline{M}$, centered at $P_i^{M} \in D_M$ extends to a holomorphic function $z_i : U \subset X_0 \to D^*(\epsilon)$, for some $\epsilon > 0$ and open neighbourhood $U$ of $P_i^{M}$. 
 	Thus, the conformal completion of $D^*(\epsilon)$ induces the compactification $X$ of $X_0$ such that $X \setminus X_0 = \cup_{i = 1}^{m} \Im (\sigma_i)$ for some non-intersecting holomorphic functions $\sigma_i : D(c)^N \to X$.
 	Also, trivially, the action of $\rm{Aut}(\textbf{M})$ over $H^1(\overline{M}, T^{1,0} \overline{M} \otimes \mathscr{O}_{\overline{M}}(- D_M))$ induces the action on $X$, which preserves $\sigma_1, \ldots, \sigma_m$.
 	\par By Serre duality, for $\textbf{M} := (\overline{M}, D_M) \in \mathscr{M}_{g, m}$, we have the isomorphism
 	\begin{equation}
 		 H^1(\overline{M}, T^{1,0} \overline{M} \otimes \mathscr{O}_{\overline{M}}(- D_M)) 
 		 \simeq
 		 H^0(\overline{M}, \omega_{\overline{M}}^{2} \otimes \mathscr{O}_{\overline{M}}(D_M)).
 	\end{equation}
 	By the uniformization theorem, there is the unique hyperbolic metric $g^{TM}_{\rm{hyp}}$ of constant scalar curvature $-1$ over $M$ with cusps at $D_M$.
 	We endow the space $H^0(\overline{M}, \omega_{\overline{M}}^{2} \otimes \mathscr{O}_{\overline{M}}(D_M)) \subset \ccal^{\infty}(\overline{M}, \omega_{\overline{M}}^{2} \otimes \mathscr{O}_{\overline{M}}(D_M))$ with the $L^2$-scalar product from (\ref{defn_l2_scal}). This defines the Kähler metric on $\modulcomp_{g, m}$, which is called the \textit{Weil-Petersson metric}. 
 	The \textit{Weil-Petersson form}, which we denote by $\omega_{WP}$, is the Kähler form associated with the Weil-Petersson metric.
 	\par By the uniformization theorem, the relative canonical line bundle $\omega_{g, m}$ of $\Omega$ can be endowed with the Hermitian metric $\norm{\cdot}_{g, m}^{\omega, \rm{hyp}}$ over $\univcurv_{g, m}$ in such a way that the restriction of this metric over each fiber induces the hyperbolic Kähler metric of constant scalar curvature $-1$ on the fibers. 
 	By Teichmüller theory, this metric is smooth over $\univcurv_{g, m}$.
 	Let $D_{g, m}$ be the divisor in $\univcurvcomp_{g, m}$, which is formed by the fixed points of the fibers.
 	We endow the twisted canonical line bundle $\omega_{g, m}(D)$ (cf. (\ref{def_rel_can_modul})) with the induced norm $\norm{\cdot}_{g, m}^{\rm{hyp}}$ as in Construction \ref{const_norm_div}.
 	The following interpretation of $\omega_{WP}$ lies in the core of our applications.
 	\begin{thm}[{Wolpert \cite[Corollary 5.11]{WolpChForm86}, (cf. \cite[Corollary 5.2.2]{FreixTh})}]\label{thm_wolp_identity}
 		The following identity of smooth forms over $\modul_{g, m}$ holds:
 		\begin{equation}
 			\omega_{WP} = \pi^2  \Omega_* \Big[c_1 \big(\omega_{g, m}(D), \norm{\cdot}_{g, m}^{\rm{hyp}} \big)^2 \Big]^{[4]}.
 		\end{equation}
 	\end{thm}
	\par Now, to describe the local structure of $\modulcomp_{g, m}$, $\univcurvcomp_{g, m}$ near the boundary, we describe the deformations of a pointed complex curve 
	\begin{equation}\label{eq_rbar_fix}
	(\overline{R}, D_R) \in \partial \modul_{g, m}, \quad D_R \subset \overline{R}, \quad \# D_R < \infty
	\end{equation}
	with double-point singularities $\Sigma_{R} = \{ q_1, \ldots, q_l \} \subset \overline{R}$, $\Sigma_R \cap D_R = \emptyset$.
	\par We write$\overline{R} \setminus \Sigma_{R} = \cup_{i=1}^{k} R_i$, for some open Riemann surfaces  $R_i$. 
	Then $D_R$ induces the marked points $D_{R_i}^{0}$ on $R_i$.
	We compactify each $R_i$ to $\overline{R}_i$ by filling the created punctures, appearing after deletion of the nodes, and denote
	\begin{equation}\label{eq_dri_defn}
		 R^0 = \overline{R} \setminus (\Sigma_R \cup D_R), \quad D_{R_i} := D_{R_i}^{0} \cup (\overline{R}_i \setminus R_i).
	\end{equation}
	Suppose that for any $i = 1, \ldots, k$, the marked surfaces $(\overline{R}_i, D_{R_i})$ are stable. We describe small deformations of $(\overline{R}, D_R)$ in terms of small deformations of $(\overline{R}_i, D_{R_i})$ and so-called \textit{plumbing construction}, which we are going to describe now.
	\par 
	For every $j = 1, \ldots, l$, the complex curve $\overline{R} \setminus \{ q_j \}$, has a couple of punctures $\{a_j, b_j\}$ at the place of $q_j$. For the punctures $q_j$, we consider
	\begin{enumerate}
		\item A neighbourhood $A_j$ of the puncture $a_j$, biholomorphic to a punctured disc. We denote by $U_j$ the conformal completion of $A_j$, obtained by formally adding $a_j$. Then $U_j$ is biholomorphic to an open disc. 
		Let $F_j : U_j \to \comp$ be a holomorphic coordinate mapping with $F_j(a_j) = 0$;
		\item Similarly, a neighbourhood $B_j$ of the puncture $b_j$, its conformal completion $V_j$ and a coordinate mapping $G_j : V_j \to \comp$ satisfying $G_j(b_j) = 0$;
		\item A small complex parameter $t_j \in \comp$.
	\end{enumerate}
	We suppose that the sets $A_j$ and $B_j$ are mutually disjoint for $j = 1, \ldots, l$, and they are disjoint from $D_{R}$. 
	Let $c > 0$ be such that $D(c) \subset \comp$ is contained in $\Im (F_j), \Im (G_j)$, for all $j$. Assume that $|t_j| < c^2$, for all $j$. We denote $t = (t_1, \ldots, t_l) \in D(c^2)^{l}$. 
	For $d = (d_1, \ldots, d_l) \in D(c)^l$, we note
	\begin{equation}
		R^{d, *} = R^0 \setminus \cup_{j = 1}^{l} \big( \{ |F_j| \leq |d_j| \} \cup \{ |G_j| \leq  |d_j| \} \big).
	\end{equation}
	\par Consider the equivalence relation on points of $R^{t/c, *}$ generated by: $p \sim q$ if there exists $j  = 1, \ldots, l$, such that
	$|t_j|/c \leq |F_j(p)| \leq c $, $|t_j|/c \leq |G_j(q)| \leq c$ and $F_j(p)G_j(q) = t_j$.
	 Form the identification space $\overline{R}_t = R^{t/c, *} / \sim$. By the construction, $D_{R}$ induces the set of points $D_{R_t}$ on $\overline{R}_t$. We say that the compact pointed complex curve $(\overline{R}_t, D_{R_t})$ is the \textit{plumbing construction} for $(R, D_{R})$ associated with the \textit{plumbing data } $\{ (U_j, V_j, F_j, G_j, t_j) \}_j$. Trivially, we see that a set $X := \cup_{t \in D(c^2)^l} R_t$ can be endowed with a structure of a complex manifold, for which $\pi : X \to D(c^2)^l$ is a proper holomorphic map of codimension 1.
	 \par 
	 Now let's present a construction which combines the deformations using Beltrami differentials and the plumbing families.
	\begin{const}\label{const_univ_family}
	\begin{sloppypar}
	Let $(\overline{R}, D_R)$,  $\Sigma_{R}$, $(\overline{R}_i, D_{R_i})$, $i = 1, \ldots, k$ be as in (\ref{eq_rbar_fix}), (\ref{eq_dri_defn}).
	 Choose a plumbing data  $\{ (U_j, V_j, F_j, G_j, t_j) \}_j$. 
	 Observe that one can take $U_j, V_j$ so small so that there are Beltrami differentials $\nu_1, \ldots, \nu_N$ such that each of them is compactly supported in exactly one connected component of $\overline{R} \setminus \cup_j (U_j \cup V_j)$ and the associated cohomology classes $[\nu_1], \ldots, [\nu_N]$ form a basis in $\oplus_{i = 1}^{k} H^1(\overline{R}_i, T^{1,0}\overline{R}_i \otimes \mathscr{O}_{\overline{R}_i}(-D_{R_i}))$. 
	To simplify the exposition, we suppose that $k = 1$, i.e. that $R^0$ is connected.
	Let $\nu(s)$, $s \in D(c)^N$ be defined as in (\ref{eqn_nu_s_defn}) for $c$ small enough.
	\par Let $\Gamma$ be a Fuchsian group such that $R^0$ is isomorphic to the quotient $\Gamma \setminus \hh$. 
	We write $\Gamma^s := f^{\nu(s)} \Gamma (f^{\nu(s)})^{-1}$ for $ f^{\nu(s)}$ as in (\ref{eq_ahl_thm}) and define a Riemann surface $R^0_{s} := \Gamma^s \setminus \hh$.
	Observe that since the support of $\nu(s)$ is contained in $\overline{R} \setminus \cup_j (U_j \cup V_j)$, by (\ref{eq_ahl_thm}),the coordinates $F_j, G_j$ induce holomorphic charts on $R^0_{s}$. 
	We can complete $R^0_{s}$ by adding points representing $\{ F_j = 0 \}$ and $\{ G_j = 0 \}$. 
	By identifying those pairs of points, we get a compact complex curve $\overline{R}_s$. 
	The set of singular points of $\overline{R}_s$ is in the obvious bijection with $\Sigma_R$.
	Now, again by the fact that the support of $\nu(s)$ is contained in $\overline{R} \setminus (\cup_j (U_j \cup V_j))$, the plumbing data $\{ (U_j, V_j, F_j, G_j, t_j) \}_j$ on $\overline{R}$ induces the plumbing data on $\overline{R}_s$. Thus, for $t \in D(c)^l$, we form a complex curve $\overline{R}_{s, t}$.
	\end{sloppypar}
	\end{const}
	\begin{prop}[{Wolpert \cite[p. 434]{Wolp90}}]\label{prop_univ_family}
	\begin{sloppypar}
	Construction \ref{const_univ_family} has the following properties:
		\par a) The complex parameters $(s, t)$ in $S := D(c)^{N + l}$ are local coordinates for local manifold covers of $\modulcomp_{g, m}$ in a neighbourhood of a point defined by $(\overline{R}, D_R)$. The divisor of singular curves $\Delta$ is given by $\{ z_{N + 1} = 0 \} + \cdots + \{ z_{N + l} = 0 \}$, thus, has normal crossings.
		\par b) The set $X  := \cup_{(s, t) \in S} \overline{R}_{s, t}$ can be endowed with a structure of a complex manifold such that the projection $\pi : X \to S$ is a f.s.o.
		\par c) The fixed points $D_R$ induce the holomorphic sections $\sigma_1, \ldots, \sigma_m : S \to X$.  Then $(\pi, \sigma_1, \ldots, \sigma_m )$ provides a description for local manifold covers of $\Omega : \univcurvcomp_{g, m} \to \modulcomp_{g, m}$.
	\end{sloppypar}
	\end{prop}
	
	\subsection{Pinching expansion and proof of Corollaries \ref{cor_cont_hodge}, \ref{cor_form_TZ_type}, \ref{cor_wp_cont_pot}, \ref{cor_wp_vol}, \ref{cor_del_iso}}\label{sect_pinch}
		In this section we will explain why the hyperbolic metric over the universal curve satisfies Assumptions \textbf{S2}, \textbf{S3}.
		For this, we recall the \textit{pinching expansion} of the hyperbolic metric. 
		Then we establish Corollaries \ref{cor_cont_hodge}, \ref{cor_form_TZ_type}, \ref{cor_wp_cont_pot}, \ref{cor_wp_vol}, \ref{cor_del_iso}.
		\par Pinching expansion describes the behaviour of the hyperbolic metric near the boundary of the universal curve. It compares the hyperbolic metric with so-called \textit{grafted metric}, which is more accessible for analysis.
		We follow closely the description of Wolpert \cite{Wolp90}. 
		\par 
		Let $(\overline{R}, D_R)$, $R^0$, $\Sigma_{R}$, $(\overline{R}_i, D_{R_i})$, $i = 1, \ldots, k$ be as in Section \ref{sect_orb_str}.
		Let $\{ (U_j, V_j, F_j, G_j, t_j) \}_j$ be a plumbing data for $\overline{R}$. Consider the plumbing construction $R_{s, t}$, $(s, t) \in S := D(c)^{N + l}$, $c > 0$. The grafted metric is built from the hyperbolic metric on $R^0$ and the \textit{hyperbolic metric on a cylinder}, see (\ref{defn_g_cyl1}), (\ref{defn_g_cyl2}). Let's describe this construction more precisely. 
		\begin{sloppypar} 
	Let $g^{T R^0}_{\rm{hyp}}$	 be the hyperbolic metric of constant scalar curvature $-1$ with cusps on $\cup_{i = 1}^{k} (\overline{R}_i, D_{R_i})$. 
	Let $u_j$, $v_j$, $j = 1, \ldots, l$ be some Poincaré-compatible coordinates around $\cup_{i = 1}^{k} D_{R_i}$ with respect to $g^{T R^0}$. We denote 
	\begin{equation}\label{eqn_alp_beta_defn}
			\alpha_j = \big| (F_j \circ u_j^{-1}) '(0) \big|^{-1}, \qquad \beta_j = \big| (G_j \circ v_j^{-1})'(0) \big|^{-1}.
	\end{equation}		
	We renormalize the coordinates 
	\begin{equation}\label{eq_fg_tau}
		f_j = \alpha_j F_j, g_j = \beta_j G_j, \quad \tau_j := \alpha_j \beta_j t_j.
	\end{equation}
	Trivially from Section \ref{sect_orb_str}, the plumbing construction for $(U_j, V_j, F_j, G_j, t_j)$ coincides with the plumbing construction for $(U_j, V_j, f_j, g_j, \tau_j)$.
	\end{sloppypar}	
	\par Let $\nu_{1, 0} : R^0 \to [0,1]$, $\nu_2 : \real \to [0,1]$ be smooth functions, satisfying
	\begin{align}
		& \nu_{1, 0}(x) = 
		\begin{cases} 
      		\hfill 1, & \text{ for } x \in R^{(1 + \delta)c, *}, \\
      		\hfill 0,  & \text{ for } x \in R^0 \setminus R^{(1 - \delta)c, *}. 
 		\end{cases}
 		\\
		& \nu_2(w) \, \,=  
		\begin{cases} 
      		\hfill 0, & \text{ for } w < 1/2 - 2\delta, \\
      		\hfill 1,  & \text{ for } w > 1/2 + 2\delta. 
 		\end{cases}
	\end{align}
	Since the function $\nu_{1, 0}$ is zero in the pinching collar, it induces the functions $\nu_{1, s, t} : R_{s, t} \to [0,1]$, $(s, t) \in D(c)^{N + l}$ by zero away from $R^{(1 - \delta)c, *}$.
	\par Let $0 < c, \delta < 1$ be some small real constants. Since $f_j'(0), g_j'(0) = 1$, the inner boundary of annuli $\{ c < |f_j| < 2c \}$, $\{ c < |g_j| < 2c \}$ are approximately $\{ u_j = c \}$ and $\{ v_j = c \}$ respectively.
	\par We suppose that $c > 0$ is chosen in such a way that the metric $g^{T R^0}$ is induced by 
	\begin{equation}
	\begin{aligned}
		& \frac{\imun d u_j d \overline{u}_j }{| u_j \log |u_j|^2 |^2},  
      			&& \text{ over } \big\{ |f_j| < 2 c \big\}, \\
      	& \frac{\imun d v_j d \overline{v}_j }{| v_j \log |v_j|^2 |^2}, 
      			&& \text{ over } \big\{ |g_j| < 2 c \big\}.
	\end{aligned}
	\end{equation}
	We denote by $g^{\text{Cyl}}_{j,1}$ the metric over the set (see (\ref{eq_fg_tau}))
	\begin{equation}
		\big\{ |\tau_j|^{1/2 + 2 \delta} < |f_j| < e^{2 \delta} c \big\} = \big\{ e^{- 2 \delta} |\tau_j|/c < |g_j| < |\tau_j|^{1/2 - 2 \delta} \big\},
	\end{equation}		
	induced by the Kähler form
	\begin{equation}\label{defn_g_cyl1}
		\bigg( 
			\frac{\pi}{|u_j| \log |\tau_j|}
			\bigg(
				\sin 
				\frac{\pi \log |u_j|}{\log |\tau_j|}
			\bigg)^{-1}		
		\bigg)^2
		\imun
		d u_j d \overline{u}_j.
	\end{equation}
	Similarly, we denote by $g^{\text{Cyl}}_{j,2}$ the metric  over the set  (see (\ref{eq_fg_tau}))
	\begin{equation}
		\big\{ |\tau_j|^{1/2 + 2 \delta} < |g_j| < e^{2 \delta} c \big\} = \big\{ e^{- 2 \delta} |\tau_j|/c < |f_j| < |\tau_j|^{1/2 - 2 \delta} \big\},
	\end{equation}		
	induced by the Kähler form
	\begin{equation}\label{defn_g_cyl2}
		\bigg( 
			\frac{\pi}{|v_j| \log |\tau_j|}
			\bigg(
				\sin 
				\frac{\pi \log |v_j|}{\log |\tau_j|}
			\bigg)^{-1}	
		\bigg)^2
		\imun
		d v_j d \overline{v}_j.
	\end{equation}
	The \textit{grafted metric} $g^{TR_{s, t}}_{\text{gft}}$ is given by
	\begin{equation}\label{defn_gft}
			g^{TR_{s, t}}_{\text{gft}} = 
      		\big(
      			g^{\text{Cyl}}_{j,1} (g^{\text{Cyl}}_{j,2} / g^{\text{Cyl}}_{j,1})^{\nu_2(\log |f_j| / \log |\tau_j|)}
      		 \big)^{1-\nu_{1, s, t}}
      		 (g^{T R_0}_{\rm{hyp}})^{\nu_{1, s, t}}.
	\end{equation}
	\begin{figure}[h]
		\includegraphics[width=\textwidth]{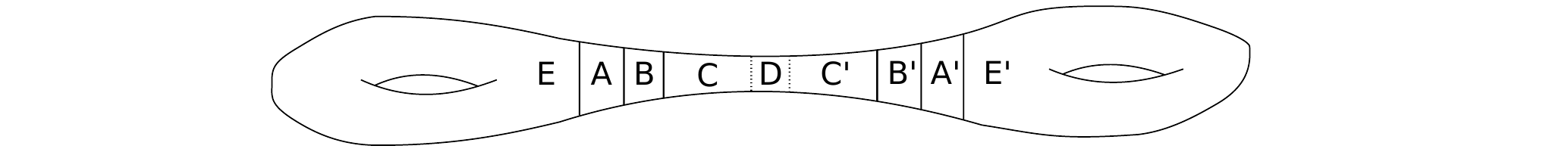}	
		\caption{The grafted metric. Over the regions $E$,$E'$, the metric $g^{TR_{s, t}}_{{\rm{gft}}}$ is isometric to $g^{TR^0}_{\rm{hyp}}$. Over the regions $A$, $A'$ the metric $g^{TR_{s, t}}_{{\rm{gft}}}$ is Poincaré-compatible with $u_j$ and $v_j$ respectively (see (\ref{reqr_poincare})). Over the regions $B$, $B'$ the metric $g^{TR_{s, t}}_{{\rm{gft}}}$ is a geometric interpolation between the Poincaré metric and (\ref{defn_g_cyl1}), (\ref{defn_g_cyl2}) respectively. Over the regions $C$, $C'$ the metric $g^{TR_{s, t}}_{{\rm{gft}}}$ is given by (\ref{defn_g_cyl1}) and (\ref{defn_g_cyl2}) respectively. Finally, over the region $D$,  the metric $g^{TR_{s, t}}_{{\rm{gft}}}$ is a geometric interpolation between (\ref{defn_g_cyl1}) and (\ref{defn_g_cyl2}). The dashed lines represent that the boundary of the region, given by the conditions $\{|f_j| = \rm{const}\}$, $\{|g_j| = \rm{const}\}$, and filled lines represent the boundary given by $\{|u_j| = \rm{const}\}$, $\{|v_j| = \rm{const}\}$.}
		\label{fig_graft}
	\end{figure}
	\begin{rem}
		If $f_j = u_j$ and $g_j = v_j$, then since $u_j v_j = \tau_j$, the metrics $g^{\text{Cyl}}_{j,2}$, $g^{\text{Cyl}}_{j,1}$ coincide over the set $\{ |\tau_j|^{1/2 + 2 \delta} < |f_j| < |\tau_j|^{1/2 - 2 \delta} \}$, and the formula for $g^{TR_{s, t}}_{\text{gft}}$ becomes simpler.
		This corresponds to the \textit{model grafting} in the terminology of Wolpert \cite{Wolp90}.
	\end{rem}
	Let's recall the pinching expansion of the hyperbolic metric. 
	Let's denote by $g^{TR_{s, t}}_{\rm{hyp}}$ the hyperbolic metric with cusps on $(R_{s, t}, D_{R_{s, t}})$.
	The following results was proved in the compact case by Wolpert \cite[Expansion 4.2]{Wolp90} and in the non-compact case by Freixas \cite[Theorem 4.3.1]{FreixTh}:
	\begin{thm}[The pinching expansion]\label{thm_wolp_exp}
		For $(s, t) \in S \setminus |\Delta|$, we have
		\begin{equation}\label{compar_g_gft}
			 g^{TR_{s, t}}_{\rm{hyp}} = g^{TR_{s, t}}_{\rm{gfh}} \Big(1 + \sum O \big( |\log |t_j||^{-2} \big) \Big),
		\end{equation}
		where the $O$-term is for the $\ccal^{\infty}$ norm over $R_{s, t} \setminus D_{R_{s, t}}$ with respect to $g^{TR_{s, t}}_{\rm{hyp}} $.
	\end{thm}
	The metrics $g^{TR_{s, t}}_{\rm{hyp}}$ induce the Hermitian norm $\norm{\cdot}_{g, m}^{\omega, \rm{hyp}}$ on $\omega_{g, m}$ over $\univcurv_{g, m} \setminus D_{g, m}$ from Section \ref{sect_orb_str}.
	By Teichmüller theory, the norm $\norm{\cdot}_{g, m}^{\omega, \rm{hyp}}$ is smooth over $\univcurv_{g, m} \setminus D_{g, m}$. 
	We denote by $\norm{\cdot}_{g, m}^{\rm{hyp}}$ the induced norm on $\omega_{g, m}(D)$. We denote by $\Sigma_{g, m}$ the set on double points singularities in $\univcurvcomp_{g, m}$. Then 
	\begin{prop}\label{prop_hyp_metr_assump}
		The hyperbolic metric $\norm{\cdot}_{g, m}^{\rm{hyp}}$ satisfies Assumptions \textbf{S2} and \textbf{S3}.
	\end{prop}
	\begin{proof}
		We recall that the goodness and continuity of the hyperbolic metric was proved for compact surfaces by Wolpert in \cite[Theorem 5.8]{Wolp90} and for non-compact by Freixas in \cite[Theorem 4.0.1]{FreixTh}. By \cite[Lemma 4.26]{BurKrKun}, any good line bundle is pre-log-log. Thus, Assumption \textbf{S2} is satisfied by $\norm{\cdot}_{X/S}^{{\rm{hyp}}}$.
		\par 
		By Proposition \ref{prop_univ_family}c), it is enough to prove that f.s.c. $(\pi: X \to S; \sigma_1, \ldots, \sigma_m; \norm{\cdot}_{X/S}^{\rm{hyp}})$ from Construction \ref{const_univ_family} satisfies Assumptions \textbf{S3}.
		\par 
		First of all, since the metric $\norm{\cdot}_{X/S}^{\rm{hyp}}$ has constant scalar curvature $-1$, we see that the coupling of $c_1(\omega_{X/S}(D), (\, \norm{\cdot}_{X/S}^{\rm{hyp}})^{2})$ with two vertical vector fields is expressed through the coupling of the fiberwise volume form. 
		By the continuity and goodness of $\norm{\cdot}_{X/S}^{\rm{hyp}}$, we deduce that the coupling of $c_1(\omega_{X/S}(D), (\, \norm{\cdot}_{X/S}^{\rm{hyp}})^{2})$ with smooth vertical vector fields is continuous over $X \setminus (\Sigma_{X/S} \cup |D_{X/S}|)$ and  has log-log growth on $X \setminus \Sigma_{X/S}$ with singularities along $|D_{X/S}|$.
		\par 	Now let's prove the fact that $\norm{\cdot}_{X/S}^{\rm{hyp}}$ has log-log growth with singularities along $\Sigma_{X/S} \cup |D_{X/S}|$. 
		First of all, by Theorem \ref{thm_wolp_exp}, it is enough to prove so for the Hermitian norm $\norm{\cdot}_{X/S}^{\rm{gft}}$ induced by $g^{TR_{s, t}}_{\rm{gft}}$ on $\omega_{X/S}(D)$.
		Trivially, we have
		\begin{equation}\label{eq_sin_bound}
			\frac{2}{\pi} x \leq \sin(x) \leq x, \qquad \text{for } x \in [0, \pi/2].
		\end{equation}
		We fix $C > 0$ such that over $D(2c)$, we have the inequalities 
		\begin{equation}\label{bound_f_j_u_j_g_j}
			\| f_j \circ u_j^{-1} \|_{\mathscr{C}^{1}} < C, \qquad \| g_j \circ v_j^{-1} \|_{\mathscr{C}^{1}} < C,
		\end{equation}
		are satisfied. 
		Then by the identity $f_j g_j = \tau_j$ and (\ref{bound_f_j_u_j_g_j}), we deduce that  there is $C_1 > 0$ such that
		\begin{equation}\label{eq_under_sin_bound}
			\Big| \log |\tau_j| -  \log |u_j| -  \log |v_j| \Big| \leq C_1.
		\end{equation}
		By (\ref{defn_g_cyl1}), (\ref{defn_g_cyl2}), (\ref{eq_sin_bound}), (\ref{bound_f_j_u_j_g_j}) and (\ref{eq_under_sin_bound}), we deduce that there is $C_2 > 0$ such that
		\begin{equation}
			\norm{du_j / u_j}_{X/S}^{\rm{gft}} 
			\leq 
			\begin{cases}
				\hfill C_2 |\log |u_j||, & \text{over} \quad \big\{ |\tau_j|^{1/2} < |g_j| < e^{2 \delta} c \big\}, \\
				\hfill C_2 |\log |v_j||, & \text{over} \quad \big\{ |\tau_j|^{1/2} < |f_j| < e^{2 \delta} c \big\},
			\end{cases}
		\end{equation}
		which implies by (\ref{eq_section_omegarel}) that $\norm{\cdot}_{X/S}^{\rm{gft}}$ has log-log growth with singularities along $\Sigma_{X/S} \cup |D_{X/S}|$. 
	\end{proof}
	Now let's explain some applications of Sections \ref{sect_pf_cont}, \ref{sect_pf_curv}. But before we drag the attention of the reader to the fact that $\modulcomp_{g, m}$ is an orbifold. However, since all the theorems of this article are local, they can be applied in an orbifold chart, and the final statements continues to hold for families of complex curves over an orbifold.
	\begin{proof}[Proof of Corollaries \ref{cor_cont_hodge}, \ref{cor_form_TZ_type}.]
		It follows directly from Theorems \ref{thm_cont}, \ref{thm_curv} and Proposition \ref{prop_hyp_metr_assump}.
	\end{proof}
	\begin{proof}[Proof of Corollary \ref{cor_wp_cont_pot}.]
		It follows from Theorem \ref{thm_wolp_identity}, Corollary \ref{cor_cont_pot} and  Proposition \ref{prop_hyp_metr_assump}.
	\end{proof}
	\begin{proof}[Proof of Corollary \ref{cor_wp_vol}.]
		The decomposition (\ref{eq_wolp_decomp}) follows directly from Theorems \ref{thm_cont}2, \ref{thm_wolp_identity}, Proposition \ref{prop_hyp_metr_assump} and (\ref{eq_bc_two_elem}).
		Now let's prove the identity (\ref{eq_wolp_volume}). 
		From (\ref{eq_wolp_decomp}), it is enough to prove that for $N := 3g - 3 + m$ and any $i = 1, \ldots, N$, we have
		\begin{equation}\label{eq_int_vol_0}
			\int_{\modul_{g, m}} \alpha^{N-i} (d \beta)^i = 0.
		\end{equation}
		Since $d \beta$ has log-log growth and $\alpha$ is smooth, we have
		\begin{equation}\label{eq_int_vol_1}
			\int_{\modul_{g, m}} \alpha^{N-i} (d \beta)^i 
			= \lim_{\epsilon \to 0} \int_{\modul_{g, m} \setminus B(\partial \modul_{g, m}, \epsilon)} \alpha^{N-i} (d \beta)^i,
		\end{equation}
		where $B(\partial \modul_{g, m}, \epsilon)$ is an $\epsilon$-tubular neighbourhood of $\partial \modul_{g, m}$ in $\modulcomp_{g, m}$.
		By Stokes theorem
		\begin{multline}\label{eq_int_vol_2}
			 \int_{\modul_{g, m} \setminus B(\partial \modul_{g, m}, \epsilon)} \alpha^{N-i} (d \beta)^i 
			 \\
			 =
			 \int_{\modul_{g, m} \setminus B(\partial \modul_{g, m}, \epsilon)} d \Big( \beta \alpha^{N-i} (d \beta)^{i- 1} \Big)
			 =
			 \int_{\partial B(\partial \modul_{g, m}, \epsilon)} \beta \alpha^{N-i} (d \beta)^{i-1}.
		\end{multline}
		Trivially, for any $k \in \nat$, $p \in \nat$, we have the following identity 
		\begin{equation}\label{eq_int_vol_222}
			\lim_{\epsilon \to 0} \int_{\substack{|z_0| = \epsilon \\ |z_1|, \ldots, |z_k| = \epsilon}} \frac{\log |\log |z_0||^p |dz_0|}{|z_0 \log |z_0||} \prod_{i = 1}^{k} \frac{\imun \log |\log |z_i||^p dz_i d\overline{z}_i}{|z_i \log |z_i||^2} = 0
		\end{equation}
		As $\beta$, $d \beta$ have log-log growth and $\alpha$ is smooth, similarly to \cite[Proposition 1.2]{MumHirz}, by (\ref{eq_int_vol_222}):
		\begin{equation}\label{eq_int_vol_3}
			\lim_{\epsilon \to 0}  \int_{\partial B(\partial \modul_{g, m}, \epsilon)} \beta \alpha^{N-i} (d \beta)^{i-1} = 0
		\end{equation}
		From (\ref{eq_int_vol_1}), (\ref{eq_int_vol_2}) and (\ref{eq_int_vol_3}), we deduce (\ref{eq_int_vol_0}).
	\end{proof}
	\begin{proof}[Proof of Corollary \ref{cor_del_iso}]
		As it was proved by Deligne in \cite[Proposition 8.5]{DeligDet} for $m=0$, and Freixas in \cite[Theorem 5.1.3, Corollary 5.1.4]{FreixTh} for $m \in \nat$, the Hermitian norm $\norm{\cdot}^{\rm{Del}}_{g, m}$ on the Deligne-Weil product $\langle \omega_{g, m}(D), \omega_{g, m}(D) \rangle$ is smooth over $\modul_{g, m}$, and we have
		\begin{equation}\label{eq_ch_dw}
			c_1 \Big( \langle \omega_{g, m}(D), \omega_{g, m}(D) \rangle, \big( \norm{\cdot}^{\rm{Del}}_{g, m} \big)^2 \Big) 
			=
			\pi^{-2} \omega_{WP}.
		\end{equation}
		By another result of Freixas, \cite[Theorem 5.1.3]{FreixTh}, the Hermitian norm $\norm{\cdot}^{\rm{Del}}_{g, m}$ on the Deligne-Weil product $\langle \omega_{g, m}(D), \omega_{g, m}(D) \rangle$ is nice over $\modulcomp_{g, m}$, with singularities along $\partial \modul_{g, m}$. Thus, by Remark \ref{rem_nice_chern}, the first Chern form is well-defined as a current, and from Proposition \ref{prop_pot_th}, (\ref{eq_ch_dw}), we have the following identity of currents over $\modulcomp_{g, m}$
		\begin{equation}\label{eq_ch_dw2}
			c_1 \Big( \langle \omega_{g, m}(D), \omega_{g, m}(D) \rangle, \big( \norm{\cdot}^{\rm{Del}}_{g, m} \big)^2 \Big) 
			=
			\pi^{-2} \big[ \omega_{WP} \big]_{L^1}.
		\end{equation}
		By Corollary \ref{cor_form_TZ_type} and (\ref{eq_ch_dw2}), the norm of the isomorphism (\ref{eq_del_isom}) is a pluriharmonic function over $\modulcomp_{g, m}$. As $\modulcomp_{g, m}$ is compact, we deduce that it is a constant, which finishes the proof.
	\end{proof}
		\bibliographystyle{abbrv}

\Addresses

\end{document}